\documentclass[11pt]{amsart}
\usepackage{amsmath,amsfonts,amssymb,mathrsfs}
\usepackage{graphicx,extpfeil}
\usepackage{epsfig}
\usepackage{indentfirst, latexsym, enumerate}
\usepackage{float}
\usepackage{subfigure}

\usepackage{extpfeil}
\usepackage{colortbl}
\usepackage{caption}
\usepackage{bm}
\usepackage{tikz}
\usetikzlibrary{matrix,shapes,arrows,positioning,chains}

\title[Asymptotic partial and complete phase-locking]{Asymptotic phase-locking dynamics and critical coupling strength for the Kuramoto model}

\author[Ha]{Seung-Yeal Ha}
\address[Seung-Yeal Ha]{\newline Department of Mathematical Sciences and Research Institute of Mathematics \newline Seoul National University, Seoul 08826, and \newline
Korea Institute for Advanced Study, Hoegiro 87, Seoul, 02455, Republic of Korea}
\email{syha@snu.ac.kr}

\author[Ryoo]{Sang Woo Ryoo}
\address[Sang Woo Ryoo]{\newline Mathematics Department, Princeton University, 
\newline Princeton, New Jersey 08544-1000, United States} \email{sryoo@math.princeton.edu}

\newtheorem{theorem}{Theorem}[section]
\newtheorem{lemma}{Lemma}[section]
\newtheorem{corollary}{Corollary}[section]
\newtheorem{proposition}{Proposition}[section]
\newtheorem{example}{Example}[section]
\newtheorem{remark}{Remark}[section]
\newtheorem{conjecture}{Conjecture}[section]

\newtheorem{definition}{Definition}[section]

\newcommand{\bbr}{\mathbb R}
\newcommand{\bbc}{\mathbb C}
\newcommand{\bbs}{\mathbb S}
\newcommand{\bbz}{\mathbb Z}
\newcommand{\bbt} {\mathbb T}
\newcommand{\bbn} {\mathbb N}

\newenvironment{sproof}{%
  \proof}{\endproof}

\begin{document}
\tikzstyle{block} = [rectangle, draw, 
    text width=15em, text centered, rounded corners, minimum height=3em]
\tikzstyle{line} = [draw, -latex']

\date{\today}

\subjclass{15B48, 92D25} \keywords{Complete phase-locking, complete synchronization, the Kuramoto model, order parameters, partial phase-locking, phase-locked state}

\thanks{Acknowledgment: The work of S.-Y. Ha is partially supported by National Research Foundation of Korea Grant (NRF-2017R1A2B2001864) funded by the Korea Government}

\begin{abstract}
We study the asymptotic clustering (phase-locking) dynamics for the Kuramoto model. For the analysis of emergent asymptotic patterns in the Kuramoto flow, we introduce the pathwise critical coupling strength which yields a sharp transition from partial phase-locking to complete phase-locking, and provide nontrivial upper bounds for the pathwise critical coupling strength. Numerical simulations suggest that multi- and mono-clusters can emerge asymptotically in the Kuramoto flow depending on the relative magnitude of the coupling strength compared to the sizes of natural frequencies. However, theoretical and rigorous analysis for such phase-locking dynamics of the Kuramoto flow still lacks a complete understanding, although there were some recent progress on the complete synchronization of the Kuramoto model in a sufficiently large coupling strength regime \cite{D-B1, D-B2, H-K-R}. In this paper, we present sufficient frameworks for partial phase-locking of a majority ensemble  and the complete phase-locking in terms of the initial phase configuration, coupling strength and natural frequencies. As a by-product of our analysis, we obtain nontrivial upper bounds for the pathwise critical coupling strength in terms of the diameter of natural frequencies, initial Kuramoto order parameter and the system size $N$. We also show that phase-locked states whose order parameters are less than $N^{-\frac{1}{2}}$ are linearly unstable. 
\end{abstract}
\maketitle \centerline{\date}


\section{Introduction}
\setcounter{equation}{0}
The purpose of this paper is to continue the studies begun in a series of papers \cite{C-H-J-K, C-H-K-K, C-S, D-X, D-B1, H-K-R, H-K-R2} for the quantitative analysis on the emergent dynamics for the Kuramoto model.  After Huygens's seminal observation in the middle of the seventeenth century on the asynchronization of two pendulum clocks hanging on a common bar, synchronous phenomena have been often observed in our nature and have been reported in scientific literature \cite{B-B, P-R-K}. However, its systematic study based on mathematical models has been initiated by the two pioneers  A. Winfree and Y. Kuramoto in the 1960s and 1970s in  \cite{Ku1, Ku2, Wi}. After Winfree and Kuramoto's seminal works, many phenomenological and mechanical synchronization models were proposed in the biology, engineering and statistical physics communities. Amongst others, in this paper we focus on the Kuramoto model, which serves as a prototype model in the study of synchronization. In \cite{Ku1}, Kuramoto introduced a simple first-order dynamical system for the synchronization of an ensemble of weakly coupled phase oscillators, and showed that the corresponding mean-field kinetic equation exhibits a phase-transition like phenomenon from the disordered state to the ordered state, as the coupling strength increases from zero to some large value \cite{A-B, A-R, B-C-M, B-D-P, B-N-S, B-S,   C-H-J-K, C-H-K-K, C-D, D-A, D-X, Er, Ha, M-S1, M-St, M-S2, P-R-K, St, V-W, V-M0, V-M1, V-M2, W-S}. Thanks to this phase-transition like phenomenon in the mean-field kinetic Kuramoto equation, the Kuramoto model has been extensively investigated in the physics community, particularly in statistical physics. In this paper, we focus on the Kuramoto model, not passing the system size to infinity. Numerical simulations suggest that for a given generic initial phase configuration, the ensuing Kuramoto flow displays the following asymptotic patterns, as we increase the coupling strength:
\begin{equation}\label{situation}
\mbox{Incoherent state} \quad \Longrightarrow \quad \mbox{Partial phase-locking} \quad \Longrightarrow \quad \mbox{Complete phase-locking}.
\end{equation}
The second phase-transition from partial phase-locking to complete phase-locking has been studied in several papers \cite{C-S, D-B1, D-B2, J-M-B}. In contrast, the first phase-transition, i.e., the transition from the incoherent states to partial phase-locking, has not been much investigated except for the infinite dimensional case (see recent survey papers \cite{D-B2, H-K-P-Z}). 

Let us now go into the specifics of the Kuramoto model. Kuramoto oscillators can be visualized as rotors, indexed from $1$ through $N$, moving on the unit circle $\bbs^1$. Let $z_i =  e^{{\mathrm i} \theta_i} $ be the position of the $i$-th rotor, and let $\theta_i$ and ${\dot \theta}_i$ denote the phase and frequency of the $i$-th oscillator, respectively. Then, the dynamics of Kuramoto oscillators is governed by the following Cauchy problem: 
\begin{equation} \label{Ku}
\begin{cases}
\displaystyle \dot{\theta}_{i} = \nu_i +
\frac{\kappa}{N}\sum_{j=1}^{N}\sin(\theta_{j} - \theta_{i}), \quad t > 0, \\
\theta_i(0) = \theta_{i}^0,
\end{cases} \quad 1\leq i \leq N,
\end{equation}
where $\kappa$ is a nonnegative coupling strength, and $\nu_i$
represents the quenched natural frequency of the $i$-th oscillator. Throughout the paper, we will use the following simplified notation:
\[ {\mathcal N} := \{1, \cdots, N \}, \quad \Theta := (\theta_1, \cdots, \theta_N), \quad  \dot{\Theta} := ({\dot \theta}_1, \cdots, {\dot \theta}_N), \quad 
\Omega := (\nu_1, \cdots, \nu_N),
\]
and we set the diameters for the whole configuration:
\begin{align}
\begin{aligned} \label{New-0}
& D(\Theta) := \max_{1 \leq i, j \leq N} |\theta_i - \theta_j|, \quad D({\dot \Theta}) := \max_{1 \leq i, j \leq N} |{\dot \theta}_i - {\dot \theta}_j|, \\
& D(\Omega) :=  \max_{1 \leq i, j \leq N} |\nu_i - \nu_j|, \quad  ||\Theta||_{\infty} := \max_{1 \leq i \leq N} |\theta_i|.
\end{aligned}
\end{align}
Similarly, for any subset $\mathcal{A}\subset\mathcal{N}$ we set
\[ \Theta_\mathcal{A} := (\theta_i)_{i\in\mathcal{A}},\quad  D(\Theta_\mathcal{A}) := \max_{i,j\in\mathcal{A}} |\theta_i - \theta_j|, \quad  \Omega_\mathcal{A} := (\nu_i)_{i\in\mathcal{A}}, \quad D(\Omega_\mathcal{A}) :=  \max_{ i, j\in\mathcal{A}} |\nu_i - \nu_j|.
\]
Also, we will sometimes identify the set of indices with the set of oscillators for the sake of convenience, i.e., we will identify the index $i$ with the $i$-th oscillator from time to time. 

One immediate property of \eqref{Ku} is a conservation law: if we sum \eqref{Ku} over $i$ and use the oddness of $\sin \theta$, we obtain
\begin{equation}\label{conservation}
\frac{d}{dt} \Big( \sum_{i=1}^{N} \theta_i(t)  -t \sum_{i=1}^{N} \nu_i \Big) = 0, \quad t > 0.
\end{equation}
One consequence of this conservation law is that if $\sum_{i=1}^N \nu_i\neq 0$, then the Kuramoto model \eqref{Ku} cannot have an equilibrium. However, the Kuramoto model \eqref{Ku} has another related property, namely its Galilean invariance: \eqref{Ku} is invariant under Galilean transformations such as
\begin{equation}\label{galilean}
\theta_i\mapsto \theta_i-t\cdot \frac{1}{N}\sum_{j=1}^{N} \nu_j,\quad \nu_i\mapsto\nu_i-\frac{1}{N}\sum_{j=1}^{N} \nu_j,\quad i=1,\cdots,N.
\end{equation}
We may assume that we have already taken the transformation \eqref{galilean}. (This does not affect important parameters such as $D(\Theta)$, $D(\Omega)$, $R(\Theta)$ and $\Delta(\Theta)$, which are defined in \eqref{New-0} and Section \ref{sec:2}.) The advantage of this assumption is that the average of the natural frequencies is zero:
\begin{equation} \label{zero-nat}
\sum_{i=1}^{N} \nu_i = 0,
\end{equation}
and the conservation law \eqref{conservation} now becomes conservation of total phase:
\begin{equation}\label{conservation-zero}
\sum_{i=1}^N \theta_i=\sum_{i=1}^N \theta_i^0.
\end{equation}
Therefore, it makes sense to discuss equilibria of the transformed variables \eqref{galilean}. Equilibria of \eqref{Ku} under the Galilean transformation \eqref{galilean}, i.e., equilibria relative to a frame rotating with the average phase velocity $\frac{1}{N}\sum_{i=1}^N \nu_i$,  are called phase-locked states, and convergence to such equilibria is called (asymptotic) phase-locking. A detailed definition will be given in Definition \ref{D2.1}.

Another common observation is that the right-hand side of \eqref{Ku} is $2\pi$-periodic, and thus system \eqref{Ku} can be regarded either as a dynamical system of the variables $\{e^{{\mathrm i}\theta_i}\}_{i=1}^N$ on the $N$-torus $\bbt^N$, or as a dynamical system of the variables $\{\theta_i\}_{i=1}^N$ on the Euclidean space $\bbr^N$. Formally, we may view the dynamical system on the Euclidean space $\bbr^N$ as a lift of the dynamical system on $\bbt^N$. In this paper, for convenience of argument, we will use the formulation on $\bbr^N$ in Sections \ref{sec:2} to \ref{sec:5}, and that on $\bbt^N$ in Section \ref{sec:6}. 

Numerical simulations suggest that emergence of phase-locking is possible for generic initial configurations, as long as the coupling strength satisfies $\kappa >D(\Omega)$. A plethora of research was focused on the rigorous verification of this numerical observation. In \cite{C-S, D-B1, D-B2, H-H-K, H-K-P, H-L-X, J-M-B}, the emergence of phase-locked states has been studied for some restricted class of initial (phase) configurations, particularly for phase configurations confined in a half circle.  Recently, Dong and Xue \cite{D-X} employed the gradient flow formulation for  the Kuramoto model and, using this alternative formulation, they showed that in the ensemble of identical oscillators, i.e., $\nu_i=\nu_j$ for all indices $i$ and $j$, the Kuramoto flow exhibits phase-locking for all initial configurations and any positive coupling strength. For the nonidentical case, the authors in \cite{H-K-R} used a comparison argument between the identical and nonidentical case to establish asymptotic phase-locking in the nonidentical case for sufficiently strong coupling. However, such an argument was not able to give an explicit critical coupling strength independent of the system size $N$ and initial data $\Theta^0$. Hence, some natural questions to think of are as follows.

\smallskip

\begin{itemize}
\item
(Q1): Can we describe a sufficient coupling strength for the complete phase-locking explicitly in terms of $\Omega, \Theta^0$ and $N$?

\vspace{0.2cm}

\item
(Q2): Can we describe the emergence of partial phase-locking rigorously?

\vspace{0.2cm}

\item
(Q3): Can we find a minimal coupling strength leading to complete phase-locking from a generic initial phase configuration?
\end{itemize}

\smallskip

In this paper, we answer the first and second questions, i.e., we provide a sufficient coupling strength leading to complete phase-locking in terms of $D(\Omega)$ and $R_0$ (to be defined in Section \ref{sec:2.1}), and a sufficient framework for the formation of partial phase-locking ensembles.

The novelty of this paper is three-fold. First, we provide a geometric condition on the initial phase vector $\Theta^0$ that guarantees partial phase-locking, which not only explains the first phase-transition in \eqref{situation}, but also explains the second phase-transition in \eqref{situation} as a special case. More precisely, we show that for given sets of indices $\mathcal{A}\subset\mathcal{B}\subset\mathcal{N}$, if parameters $\gamma,\ell$ and $\kappa$ satisfy
\[
\frac{1}{2} < \gamma \leq 1, \quad \ell \in\left(0,2\cos^{-1} \Big( \frac{1}{\gamma} - 1 \Big) \right), \quad \kappa > C(\gamma, \ell) D(\Omega_\mathcal{B}),
\]
then the ensemble ${\mathcal A}$ forms a well-ordered stable $\gamma$-ensemble arranged in accordance with the order of natural frequencies, and the ensemble ${\mathcal B}$ becomes partially phase-locked (see Theorem \ref{T3.1} for the precise statement, and  Definitions \ref{D2.1} and \ref{D3.1} for the meaning of the jargons $\gamma$-ensemble, stable $\gamma$-ensemble and partial and complete phase-locking). 

Second, we present an improved result on complete phase-locking. So far, the best-known result about the coupling strength which guarantees asymptotic phase-locking for generic initial data is that it exists and is finite for generic initial data (see Proposition \ref{P2.2}). We improve this result by exhibiting an explicit coupling strength, namely $1.6\frac{D(\Omega)}{R_0^2}$ (see Theorem \ref{T3.2}). This has the advantage that it is free of the system size, and hence this result can be lifted to the mean-field kinetic regime via the uniform-in-time mean-field limit $N \to \infty$.

Third, we show that the region consisting of state vectors $\Theta$ whose order parameters are less than $\frac{1}{\sqrt{N}}$ is unstable in some sense for any positive coupling strength, and using this result, we identify another sufficient coupling strength  $1.6ND(\Omega)$ for complete phase-locking. Although this coupling strength has the drawback that it directly depends on $N$, it has the advantage that it is independent of the initial configuration once we fix $N$ and $\Omega$. \newline

The rest of this paper is organized as follows. In Section \ref{sec:2}, we review a theoretical minimum for synchronization, namely the concepts of partial phase-locking and complete phase-locking, the order parameters, the gradient flow formulation of the Kuramoto model, and the state-of-the-art results on complete phase-locking. 
In Section \ref{sec:3}, we summarize our main results and compare them with earlier results. In Section \ref{sec:4}, we present geometric conditions on the phase vector $\Theta^0$ that ensure partial phase-locking and discuss its detailed asymptotic behavior. In Section \ref{sec:5}, we provide a proof of our improved result on the emergence of complete phase-locking, namely that $1.6\frac{D(\Omega)}{R_0^2}$ is a sufficient coupling strength. In Section \ref{sec:6}, we analyze the divergence of the vector field \eqref{Ku} and provide the sufficient coupling strength $1.6ND(\Omega)$. Finally, Section \ref{sec:7} is devoted to a brief summary and discussion of the main results of this paper and possible future research directions. In Appendix A, we provide a  proof of Lemma \ref{L5.3}, which is the most technical part of Section \ref{sec:5}.

\section{Preliminaries} \label{sec:2}
\setcounter{equation}{0}
In this section, we recall the concepts of complete and partial phase-locking, the definition of the order parameter, and the gradient flow formulation for the Kuramoto model. \newline

As discussed in \eqref{conservation}-\eqref{conservation-zero}, phase-locked states are relative equilibria with respect to a rotating frame with the average phase velocity $\frac{1}{N} \sum_{i=1}^{N} \nu_i$. Equivalently, we may characterize phase-locked states (relative equilibria) and phase-locking using the relative phase differences as follows.
\begin{definition}  \label{D2.1}
\emph{\cite{A-B, D-B2, H-K-R}}  Let $\Theta = (\theta_1, \cdots, \theta_N)$ be a phase vector whose components satisfy the Kuramoto model \eqref{Ku} on $\bbr^N$.
\begin{enumerate}
\item
$\Theta = \Theta(t)$ is a phase-locked state of \eqref{Ku}, if all relative phase differences are constant:
\[  \theta_i(t) - \theta_j(t) = \theta_i(0) - \theta_j(0), \quad t \geq 0,  \qquad 1 \leq i, j \leq N. \] 
\item
$\Theta = \Theta(t)$ exhibits (asymptotic) complete phase-locking if the relative phase differences converge as $t \to \infty$:
\[ \exists \lim_{t \to \infty} (\theta_i(t) - \theta_j(t)), \quad 1 \leq i, j \leq N. \]
\item
$\Theta = \Theta(t)$ exhibits complete synchronization if the relative frequency differences converge as $t \to \infty$:
\[  \lim_{t \to \infty} |{\dot \theta}_i(t) - {\dot \theta}_j(t)| = 0, \quad 1 \leq i, j \leq N. \]
\end{enumerate}
\end{definition}
\begin{remark} \label{R2.1}
1. It is easy to see that complete phase-locking implies complete synchronization, i.e., under complete phase-locking, the relative frequencies must tend to zero asymptotically:
\[ \lim_{t \to \infty} |{\dot \theta}_i(t) - {\dot \theta}_j(t)| = 0, \quad 1 \leq i, j \leq N. \]
In general, the converse may not be true. However, for the Kuramoto model, since complete synchronization occurs exponentially fast, complete synchronization implies complete phase-locking. \newline

\noindent 2. From time to time, we will omit the adjective ``complete" in the phrase ``complete phase-locking". Thus, throughout the paper, phase-locking means complete phase-locking.
\end{remark}

If the coupling strength $\kappa>0$ is not large enough compared to the natural frequency diameter $D(\Omega)$, the Kuramoto model may fail to exhibit complete phase-locking and the whole Kuramoto ensemble can roughly be classified into three types of oscillators: {\it synchronizing oscillators, drifting oscillators, and fuzzy oscillators}. The first group of synchronizing oscillators consists of oscillators $\theta_i$ whose natural frequencies are much smaller than $\kappa$. Their dynamics \eqref{Ku} is governed mostly by the attractive nonlinear term  $\frac{\kappa}{N}\sum_j \sin(\theta_j-\theta_i) $ in \eqref{Ku}  and as such they display synchronous behavior. The second group of drifting oscillators consists of oscillators $\theta_i$ whose natural frequencies are much larger than $\kappa$. Their dynamics \eqref{Ku} is governed mostly by the linear term  $\nu_i $ in \eqref{Ku}  and as such they tend to drift along the circle at about a constant rate. The last group of fuzzy oscillators consists of oscillators whose natural frequencies are comparable to $\kappa$. Whether the linear or nonlinear term in \eqref{Ku} will be dominant will depend on the strength of the nonlinear attractive term(i.e., the order parameter) at each time, and so they tend to behave like synchronizing oscillators at one time and drifting oscillators at another time. The exact, rigorous, and complete classification of the above three groups in the finite-$N$ case is an important question which might also be related to complete integrability. See \cite{St} for relevant discussion. At present, we suggest that we can identify the first group of synchronizing oscillators in the following rigorous manner:
\begin{definition}
Let $\Theta = (\theta_1, \cdots, \theta_N)$ be a phase vector whose components satisfy the Kuramoto model \eqref{Ku} on $\bbr^N$. We say that $\Theta = \Theta(t)$ exhibits (asymptotic) partial phase-locking if there exists a proper subset ${\mathcal A}$ of ${\mathcal N}$ such that 
\[ \sup_{0 \leq t < \infty} D(\Theta_{\mathcal A}(t)) < \infty. \]
\end{definition}
\begin{remark}
\noindent 1. Here, we are suggesting that  $\Theta_\mathcal{A}$ is a group of synchronizing oscillators.

\noindent 2. The reason we do not require the phase differences to converge, but only require them to be bounded, is that we do not assume anything about the behavior of $\Theta_{\mathcal{N}\backslash \mathcal{A}}$. Like the definition of partial phase-locking, we can say that  $\Theta = \Theta(t)$ exhibits weak phase-locking if the phase differences are uniformly bounded:
\[ \sup_{0 \leq t < \infty} D(\Theta(t)) < \infty. \]
Complete phase-locking trivially implies weak phase-locking, and the converse holds by Proposition \ref{NoChaos}.
\end{remark}

\vspace{0.5cm}

Next, we present two alternative formulations of the Kuramoto model, namely in a mean-field form in terms of order parameters and in the form of a potential flow, which serve complementary roles in the complete synchronization estimates.

\subsection{Order parameters} \label{sec:2.1}  Let $\Theta  = \Theta(t)$ be an $N$-phase vector whose time evolution is governed by \eqref{Ku}. Then we define real order parameters $R(\Theta)$ and $\phi(\Theta)$ by the following relation:
\begin{equation} \label{Order}
R(\Theta) e^{\mathrm{i} \phi(\Theta)} :=\frac{1}{N}\sum_{j=1}^{N}e^{\mathrm{i}\theta_j}.
\end{equation}
In other words, they correspond to the modulus and argument of the centroid $\frac{1}{N}\sum_{j=1}^{N}e^{\mathrm{i}\theta_j}$ of $\{ e^{\mathrm{i}\theta_j} \}$ in the complex plane $\bbc$.

The amplitude order parameter $R(\Theta)$ is well-defined for all $t\ge 0$, is bounded by 0 and 1, and is invariant under uniform rotation. It measures the overall ``phase coherence" of the ensemble $\Theta$. For example, $R(\Theta)=1$ corresponds to the state in which all phases are the same, i.e., complete phase synchronization:
\[ R(\Theta) = 1 \quad \iff \quad \Theta = (\alpha, \cdots, \alpha)\mod 2\pi, \quad \mbox{for some } \alpha \in \bbr, \]
whereas $R(\Theta)  = 0$ corresponds to an incoherent state in which oscillators behave independently. 

On the other hand, $\phi(\Theta)$ is well defined modulo $2\pi$ if $R(\Theta)>0$, but it is meaningless when $R(\Theta)=0$. If we suppose $R(\Theta(t))>0$ for all $t$ in some time interval $\mathcal{I}$, then it is possible to choose a branch of $\phi(\Theta(t))$ smoothly on $\mathcal{I}$.

If there is no confusion, we sometimes suppress $\Theta$-dependence on $R$ and $\phi$:
\[ R(t) := R(\Theta(t)), \qquad \phi(t) := \phi(\Theta(t)), \quad t \in \mathcal{I}. \]

We can extract various identities from \eqref{Order}. First, by comparing the real and imaginary parts of both sides of \eqref{Order}, we have
\begin{equation*} \label{order-rel-0}
 R \cos \phi =  \frac{1}{N}\sum_{j=1}^{N} \cos \theta_j, \qquad R \sin \phi =  \frac{1}{N}\sum_{j=1}^{N} \sin \theta_j.
 \end{equation*}
We can also divide both sides of the relation \eqref{Order} by  $e^{{\mathrm i} \theta_i}$ and compare real and imaginary parts to get
\begin{equation} \label{order-rel-1}
 R \cos(\phi - \theta_i) = \frac{1}{N} \sum_{j=1}^{N} \cos(\theta_j  - \theta_i), \qquad  R \sin(\phi - \theta_i)  =  \frac{1}{N} \sum_{j=1}^{N} \sin(\theta_j  - \theta_i).
\end{equation}
Likewise, we can divide \eqref{Order} by $e^{{\mathrm i} \phi}$ to get
\begin{equation}\label{order-rel-2}
R=\frac{1}{N}\sum_{j=1}^N \cos(\theta_j-\phi),\qquad 0=\frac{1}{N}\sum_{j=1}^N \sin(\theta_j-\phi).
\end{equation}
Using these relations, we can obtain
\begin{equation}\label{order-rel-3}
    R^2\stackrel{\eqref{order-rel-2}_1}{=}\frac{1}{N}\sum_{i=1}^N R\cos (\theta_i-\phi)\stackrel{\eqref{order-rel-1}_1}{=}\frac{1}{N^2}\sum_{i,j=1}^N\cos(\theta_j-\theta_i).
\end{equation}
The Kuramoto model can be rewritten in mean-field form using the relation \eqref{order-rel-1}:
\begin{equation} \label{Ku-mf}
{\dot \theta}_i  = \nu_i - \kappa R \sin(\theta_i - \phi),\quad t\in \mathcal{I},
\end{equation}
that is, the oscillators $\theta_i$ determine the ``mean-field" $R$ and $\phi$ which in turn govern the dynamics of each $\theta_i$. Furthermore, as noted in \cite{H-K-P, H-S}, we can differentiate the relation \eqref{Order} with respect to $t$, and use \eqref{Ku-mf} to obtain the coupled system:
\begin{equation}\label{order-dyn}
\begin{cases} 
\displaystyle {\dot R} = -\frac{1}{N} \sum_{i=1}^{N}  \sin (\theta_i
- \phi) \Big( \nu_i -  \kappa R \sin (\theta_i - \phi) \Big), \quad t\in \mathcal{I}, \\
\displaystyle {\dot \phi} = \frac{1}{RN} \sum_{i=1}^{N} \cos (\theta_i - \phi)
\Big( \nu_i - \kappa R \sin (\theta_i - \phi) \Big).
\end{cases}
\end{equation}

We now introduce an auxiliary  functional measuring the mean square distance between the configuration $\Theta$ on $\bbs^1$ and the straight line in $\bbc$ connecting $e^{{\mathrm i}\phi}$ and $e^{{\mathrm i}(\phi+\pi)}$:
\begin{equation} \label{New-1}
 \Delta(t) := \frac{1}{N}\sum_{k=1}^N \sin^2(\theta_k-\phi)(t), \quad t \in \mathcal{I}.  
\end{equation}
Then, we have the following growth estimate:
\begin{lemma}\label{L2.1}
Let $R$ and $\phi$ be order parameters whose dynamics is governed by the coupled system \eqref{order-dyn}. Then, we have
\[
\dot{R}\geq \kappa \sqrt{\Delta} \left( R \sqrt{\Delta} -\frac{D(\Omega)}{2\kappa}\right),\quad t\in \mathcal{I}.
\]
\end{lemma}
\begin{proof}
We expand $\eqref{order-dyn}_1$ and use $\eqref{order-rel-2}_2$ to see
\begin{align}
\begin{aligned} \label{New-1-1}
{\dot R} &= \frac{\kappa R}{N} \sum_{i=1}^{N}  \sin^2 (\theta_i- \phi) -\frac{1}{N} \sum_{i=1}^{N}  \nu_i\sin (\theta_i- \phi) \\
&= \frac{\kappa R}{N} \sum_{i=1}^{N}  \sin^2 (\theta_i- \phi) -\frac{1}{N} \sum_{i=1}^{N}  (\nu_i-\omega)\sin (\theta_i- \phi),
\end{aligned}
\end{align}
where $\omega\in\bbr$ is the real number satisfying the relation:
\[
\max_{1\leq i \leq N} |\nu_i-\omega|=\frac{D(\Omega)}{2},
\]
On the other hand, we use the Cauchy-Schwarz inequality to get 
\begin{equation} \label{New-2}
 \sum_{i=1}^N |\sin(\theta_i-\phi)| \leq  \sqrt{N} \sqrt{ \sum_{i=1}^N  |\sin(\theta_i-\phi)|^2}  = N \sqrt{\Delta}.   
\end{equation} 
We use \eqref{New-1}, \eqref{New-1-1} and \eqref{New-2} to obtain
\begin{align*}
{\dot R} &\ge \kappa R\Delta -\frac{1}{N}\sum_{i=1}^N |\nu_i-\omega|\cdot|\sin(\theta_i-\phi)| \ge \kappa R\Delta -\frac{D(\Omega)}{2}\cdot \frac{1}{N}\sum_{i=1}^N |\sin(\theta_i-\phi)|\\
&\ge \kappa R\Delta -\frac{D(\Omega)}{2} \sqrt{\Delta}
\end{align*}
We take out the common factor $\kappa \sqrt{\Delta}$ to obtain the desired differential inequality for $R$.
\end{proof}

\subsection{A gradient flow formulation} \label{sec:2.2}
In this subsection, we present a second alternative formulation of the Kuramoto model.  For a given natural frequency vector $\Omega = (\nu_1, \cdots, \nu_N)$, we can easily verify that the Kuramoto model \eqref{Ku} can be recast as a  gradient flow with analytic potential $V$ on $\bbr^N$:
\begin{equation} \label{Ku-grad}
 {\dot \Theta} = -\nabla_{\Theta} V(\Theta), \quad t > 0,
\end{equation}
where the analytic potential $V[\Theta]$ is given as follows.
\begin{equation} \label{potential}
V[\Theta] :=-\sum_{k=1}^N\nu_k \theta_k+ \frac{\kappa}{2N}\sum_{k,l=1}^N \big(1-\cos(\theta_k-\theta_l)\big).
\end{equation}
Note that the double sum in \eqref{potential} can be rewritten in terms of the order parameter $R$ as follows:
\begin{equation} \label{p-o}
 \frac{\kappa}{2N}\sum_{k,l=1}^N \big(1-\cos(\theta_k-\theta_l)\big)  =\frac{\kappa N}{2}\Big[1- \frac{1}{N^2}\sum_{k,l=1}^N \cos(\theta_k-\theta_l) \Big]\stackrel{\eqref{order-rel-3}}{=}  \frac{\kappa N}{2} (1 -R^2).
 \end{equation}
 Then, it follows from \eqref{potential} and \eqref{p-o} that 
\[ 
V[\Theta]  = - \Omega \cdot \Theta  + \frac{\kappa N}{2} (1 -R^2), \]
where $\cdot$ is the usual inner product in $\bbr^N$. The {\L}ojasiewicz gradient theorem guarantees that the Kuramoto model \eqref{Ku-grad}--\eqref{potential}, as a gradient flow with real analytic potential, has the following property regarding the asymptotic dynamics.
\begin{proposition}\label{NoChaos}
\emph{\cite{H-L-X}}
Let $\Theta = \Theta(t)$ be a uniformly bounded global solution to \eqref{Ku}--\eqref{zero-nat} in $\bbr^N$:
\begin{equation}\label{bdd}
\sup_{0 \leq t < \infty} ||\Theta(t)||_{\infty} <\infty.
\end{equation}
Then, the phase vector $\Theta(t)$ and the frequency vector
 $\dot \Theta(t)$ converge to a phase locked state and the zero vector,
 respectively, as $t\to \infty$. In other words, there exists a phase locked state
 $\Theta^{\infty}$ such that
\[ \lim_{t \to \infty} ||\Theta(t) - \Theta^{\infty}||_{\infty}  = 0 \quad \mbox{and} \quad \lim_{t \to \infty} ||{\dot \Theta}(t)||_{\infty} = 0. \]
\end{proposition}
\begin{remark} \label{R2.2}
1.In dynamical systems theory, uniform boundedness does not generally imply convergence. Thus, Proposition \ref{NoChaos}  is essentially due to the gradient flow structure of the Kuramoto flow with an analytical potential. \newline

\noindent 2. By the conservation law \eqref{conservation-zero} under \eqref{zero-nat}, \eqref{bdd} is equivalent to
\[
\sup_{0 \leq t < \infty} D(\Theta(t)) <\infty.
\]
\end{remark}
Next, we recall that finiteness of the collisions between oscillators is equivalent to asymptotic phase-locking.
\begin{theorem} \label{T2.1}
{\cite{H-K-R2}}
The Kuramoto flow $\Theta=\Theta(t)$ of \eqref{Ku} achieves asymptotic phase-locking if and only if there is at most a finite number of collisions between any pair of oscillators.
Here, a collision between $\theta_i$ and $\theta_j$ at time $t=t_0$ means that
\[\nu_i\neq \nu_j,\quad \theta_i(t_0)\equiv\theta_j(t_0) \mod 2\pi.\]
\end{theorem}
\begin{sproof}
The `only if' statement is trivial. The `if' statement is proved by using the conservation law $\sum_i \theta_i=const$ and the boundedness of the $|\theta_i-\theta_j|$'s to obtain the boundedness of $\sup_{0 \leq t < \infty} ||\Theta(t)||_{\infty}$, which allows us to use Proposition \ref{NoChaos}.
\end{sproof}

\begin{remark} \label{R2.3}
As noticed in \cite{H-K-R2}, collisions occur one-way, i.e., if $\theta_i(t_0)\equiv \theta_j(t_0) \mod 2\pi$ and $\nu_i>\nu_j$, then
\[
\dot{\theta}_i(t_0)-\dot{\theta}_j(t_0)=\nu_i-\nu_j>0.
\]
Also, if $\theta_i(t_0)\equiv \theta_j(t_0) \mod 2\pi$ and $\nu_i=\nu_j$, then by exchanging the roles of $\theta_i$ and $\theta_j$ and using the uniqueness of solutions to \eqref{Ku}, we obtain
\[
\theta_i(t)=\theta_j(t)\,\quad t\in\bbr,
\]
i.e., collision does not occur. These phenomena are all due to the fact that the coupling in \eqref{Ku} is the all-to-all coupling.
\end{remark}

\subsection{Previous results} \label{sec:2.3} In this subsection, we briefly review previous results on the emergence of asymptotic phase-locking. We first look at the special case $N=2$, which exhibits a good example of the behavior of the Kuramoto model.
\begin{example} \cite{D-B1} \label{N=2}
Consider the case $N=2$ and assume without loss of generality that $\nu_1\ge \nu_2$. Then \eqref{Ku} becomes a two-oscillator system:
\[ 
{\dot \theta}_1 = \nu_1 + \frac{\kappa}{2}  \sin(\theta_2 - \theta_1), \qquad 
{\dot \theta}_2 = \nu_2 + \frac{\kappa}{2}  \sin(\theta_1 - \theta_2).
\]
Then, the phase difference $\theta_1-\theta_2$ satisfies the Adler equation
\[
\frac{d}{dt}(\theta_1-\theta_2)=D(\Omega)- \kappa \sin (\theta_1-\theta_2).
\] 
It is easy to see that we have a dichotomy: the difference $\theta_1-\theta_2$ will not converge if $\kappa <D(\Omega)$, but will always converge if $\kappa \ge D(\Omega)$. Specifically, we have the following classification: \newline
\begin{itemize}
\item
If $\kappa > D(\Omega)$, then there are exactly two phase-locked states, namely the stable state $\theta_1-\theta_2\equiv \sin^{-1} \Big(\frac{D(\Omega)}{\kappa} \Big)~~~\mod 2\pi$ and the unstable state $\theta_1-\theta_2\equiv \pi-\sin^{-1} \Big( \frac{D(\Omega)}{\kappa} \Big)~~~\mod 2\pi$. Every initial data except for the unstable phase-locked state converges to the stable phase-locked state.

\vspace{0.5cm}

\item 
If $\kappa = D(\Omega)$, then there is exactly one phase-locked state, namely the semistable state $\theta_1-\theta_2\equiv \frac{\pi}{2}~~~\mod 2\pi$. Every initial data converges to the semistable phase-locked state.

\vspace{0.5cm}

\item
If $\kappa <D(\Omega)$, there are no phase-locked states and thus phase-locking fails.
\end{itemize}
\end{example}
\vspace{0.3cm}
Hence, for the case $N=2$, there is a sharp threshold for the coupling strength $\kappa$ that determines whether the system \eqref{Ku} will exhibit asymptotic phase-locking or not. In fact, according to numerical simulations, this seems to be true even for general $N$: for a given initial configuration $\Theta^0$, as we increase the coupling strength $\kappa$ from zero to some large number compared to $D(\Omega)$, Kuramoto's phase vector $\Theta(t)$ tends to a phase locked state as $t \to \infty$. This threshold for $\kappa$ at which phase-locking emerges is commonly referred to as the critical coupling strength. More precisely, there are two subtly different formulations of the critical coupling strength, which we distinguish as the {\it pathwise critical coupling strength} and the {\it critical coupling strength}.
\begin{definition} \label{D2.3}
Let $N \geq 2$ and let $\Omega=(\nu_1,\cdots,\nu_N)\in\bbr^N$ be a natural frequency vector.
\begin{enumerate}
\item
For an initial phase vector $\Theta^0=(\theta_1^0,\cdots,\theta_N^0) \in [-\pi, \pi)^N$, the number $\kappa_{pc} = \kappa_{pc}(\Theta^0,\Omega, N)$ defined by 
\begin{align*}
\begin{aligned}
\kappa_{pc}(\Theta^0,\Omega, N) :=\inf \{\kappa_*>0:&~\mbox{for any } \kappa >\kappa_*,~\mbox{the solution to \eqref{Ku} with initial data $\Theta^0$ exhibits} \\
& \mbox{asymptotic phase-locking} \}
\end{aligned}
\end{align*}
 is called the ``{\it pathwise critical coupling strength}". 
 
\vspace{0.2cm} 
 
\item
The number $\kappa_c = \kappa_c(\Omega, N)$ defined by 
\[
\kappa_c(\Omega, N) :=\inf \{\kappa_*>0:~\mbox{for } \kappa >\kappa_*,~~\mbox{system \eqref{Ku} admits a phase-locked state} \}
\]
is called the ``{\it critical coupling strength}". 
\end{enumerate}
\end{definition}
\begin{remark}

1. The difference between the thresholds $\kappa_{c}(\Omega, N)$ and $\kappa_{pc}(\Theta^0,\Omega, N)$ is that the critical coupling strength $\kappa_{c}(\Omega, N)$ requires simply for a phase-locked state to exist, whereas the pathwise critical coupling strength $\kappa_{pc}(\Theta^0,\Omega, N)$ requires not only that phase-locked states exist, but also that the flow starting from the specific initial data $\Theta^0$ converges to some phase-locked state. We can trivially deduce that for a  given $N, \Omega$ and $\Theta^0$, the critical coupling strength is less than or equal to the pathwise critical coupling strength:
\[ \kappa_{c} (\Omega, N) \leq \kappa_{pc}(\Theta^0, \Omega, N). \]

2. Note that Example \ref{N=2}  yields 
\[ \kappa_{pc}(\Theta^0,\Omega,2)=D(\Omega) \quad and \quad \kappa_c(\Omega, 2) = D(\Omega).  \]
\end{remark}

\vspace{0.5cm}

The critical coupling strength $\kappa_c(\Omega, N)$ has been thoroughly investigated in the literature and is quite well understood. Much research \cite{A-R, C-S, D-B1, Er, J-M-B, M-St, V-M2} have been devoted to calculating or estimating $\kappa_{c}(\Omega, N)$ for several variations of the Kuramoto model \eqref{Ku}. For example, Ermentrout \cite{Er} dealt with the infinite-dimensional case, whereas Jadbabaie, Motee, and Barahona \cite{J-M-B} dealt with the Kuramoto model on general graphs. The exact computation of $\kappa_{c}(\Omega, N)$ for our model \eqref{Ku}, where there are finitely many particles and the communication weights are given by the all-to-all coupling, has been accomplished: for example, Verwoerd and Mason \cite{V-M2} have calculated that
\begin{equation*}\label{critical-0}
\kappa_{c}(\Omega, N)=\frac{Nu_*}{\sum_{i=1}^N \sqrt{1-\frac{\nu_i^2}{u_*^2}}},
\end{equation*}
where $u_*\in[\|\Omega\|_\infty,2\|\Omega\|_\infty]$ is the unique solution of
\begin{equation*}\label{critical-1}
2\sum_{i=1}^N \sqrt{1-\frac{\nu_i^2}{u_*^2}}=\sum_{i=1}^N 1\Big /\sqrt{1-\frac{\nu_i^2}{u_*^2}}.
\end{equation*}
A crude but insightful estimate is
\begin{equation} \label{New-3}
\frac{N}{2(N-1)}D(\Omega)\le \kappa_{c}(\Omega, N)\le D(\Omega).
\end{equation}
The first inequality in \eqref{New-3} is a special case of a more general result on general symmetric networks which can be found in \cite{J-M-B}. The second inequality in \eqref{New-3} is a direct consequence of Theorem \ref{T2.2} below. Further results on $\kappa_{c}(\Omega, N)$ can be found in the survey articles \cite{D-B2, H-K-P-Z}.

Compared to the critical coupling strength $\kappa_c(\Omega,N)$, the pathwise critical coupling strength $\kappa_{pc}(\Theta^0,\Omega,N)$ is poorly understood. We ask the following natural questions: \newline
\begin{quote}
\begin{itemize}
\item
Can we explicitly compute $\kappa_{pc}(\Theta^0, \Omega, N)$ in terms of $\Theta^0$, $\Omega$ and $N$?
\item
If not, can we find upper bounds for $\kappa_{pc}(\Theta^0,\Omega, N)$ in terms of $\Theta^0$, $\Omega$ and $N$?  
\end{itemize}
\end{quote}
Numerical simulations suggest an answer to the second question as follows.
\begin{conjecture}{\cite{Ha}}\label{MainConj}
The pathwise critical coupling strength $\kappa_{pc}(\Theta^0,\Omega, N)$ is uniformly bounded over  a.e. $\Theta^0$:
\begin{enumerate}
\item (Weak version): There is a universal constant $C \geq 1$, independent of $N$, $\Theta^0$ and $\Omega$ such that
\[
\kappa_{pc}(\Theta^0,\Omega, N)\le CD(\Omega)\quad \mbox{for a.e. }\Theta^0\mbox{ and }\Omega.
\]
\item (Strong version): The constant $C$ can be taken as unity, i.e., 
\[
\kappa_{pc}(\Theta^0,\Omega, N)\le D(\Omega)\quad \mbox{for a.e. }\Theta^0\mbox{ and }\Omega.
\]
\end{enumerate}
\end{conjecture}
\begin{remark}
1. The reason why we say ``for a.e. $\Theta^0$ and $\Omega$" is because there are measure-zero counterexamples for which $\kappa_{pc}(\Theta^0,\Omega,N)=\infty$. See Example \ref{NoSync}.

\noindent 2. The advantage of an $N$-independent upper bound is that it remains valid in the mean-field limit $N \to \infty$ and so we can lift analytical results on the finite-dimensional system to the kinetic regime.

\noindent 3. See \cite{Ha} for related context behind Conjecture \ref{MainConj}, as well as its relation to different consensus models.
\end{remark}

\vspace{0.2cm}

Some research has provided estimates for $\kappa_{pc}$ for initial data near phase-locked states by investigating their linear and nonlinear stability. There is one phase-locked state that is particularly well-behaved. For the case $N=2$, we have seen in Example \ref{N=2} that when $\kappa \ge D(\Omega)$, there is a phase-locked state that is distributed on an arc of length $\sin^{-1} \Big( \frac{D(\Omega)}{\kappa} \Big)$ which has as basin of attraction the entire domain minus some measure-zero set; for arbitrary $N\ge 2$ and for $\kappa\ge D(\Omega)$, there is a similar phase-locked state on an arc of length $\sin^{-1} \Big( \frac{D(\Omega)}{\kappa} \Big)$ which has a large basin of attraction. For several variations of the Kuramoto model, this phase-locked state has been studied in \cite{A-R, B-N-S,  C-H-J-K, D-A, J-M-B, M-S1, M-S2, V-W} using mathematical tools like Lyapunov functionals, spectral graph theory, and control theory. In particular, the results in Chopra and Spong \cite{C-S}, Choi et al. \cite{C-H-J-K}, D{\"o}rfler and Bullo \cite{D-B1}, Ha, Kim and Park \cite{H-K-P}, and Ha et al.\cite{H-K-R} are the most relevant to our setting \eqref{Ku}. These papers employed the phase diameter $\displaystyle D(\Theta)$ as a Lyapunov functional, studied its temporal evolution, and, via a nonlinear Gronwall inequality for $D(\Theta)$, derived asymptotic phase-locking  for some restricted class of initial phase vectors confined in some arc of the circle. We summarize the important relevant results in the following theorem.
\begin{theorem} \label{T2.2}
\emph{\cite{C-H-J-K,  D-B1, D-B2,H-K-P}} Suppose that the coupling strength $\kappa$ and an auxiliary parameter $\kappa_e$ satisfy
\[ \kappa > \kappa_e> D(\Omega) \]
and let $\Theta = \Theta(t)$ be a solution to \eqref{Ku} such that there exists a positive time
$T \in (0, \infty)$ such that
\[ D(\Theta(T)) < \pi - {{ \arcsin \left (\frac{D(\Omega)}{\kappa_e} \right)}}. \]
Then, the following assertions hold:
\begin{enumerate}
\item
The phase diameter is bounded: there exists a finite time $T'\ge T$ such that
\[
D(\Theta(t))\le \arcsin \left (\frac{D(\Omega)}{\kappa} \right),\quad for~t\ge T'.
\]
\item
The phase vector $\Theta(t)$ approaches a phase-locked state $\Theta^\infty$ at exponential rate: there exist positive
constants $C_0(T)$ and $\Lambda = {\mathcal O}(\kappa)$ such that
\[ D({\dot \Theta}(t)) \leq C_0 \exp(-\Lambda (t-T)), \quad \mbox{for } t\ge T.\]
\item
The emergent phase-locked state $\Theta^\infty$ is unique up to $U(1)$-symmetry, and is ordered according to the ordering of their natural frequencies: there are constants $U$ and $L$ such that for any indices $i,j$ with $\nu_i\ge \nu_j$,
\[
\sin^{-1}\left(\frac{\nu_i-\nu_j}{\kappa U}\right)\le \theta_i^\infty-\theta_j^\infty \le \sin^{-1}\left(\frac{\nu_i-\nu_j}{\kappa L}\right).
\]
\end{enumerate}
\end{theorem}

\vspace{0.5cm}

In summary, Theorem \ref{T2.2} says that as long as $\kappa >D(\Omega)$, there is a unique phase-locked state confined in the open quarter-circle, which has a basin of attraction consisting of initial phase vectors distributed on an open arc of length $\left[\pi - {{ \arcsin \left (\frac{D(\Omega)}{\kappa} \right)}}\right]$. However, numerical simulations suggest that this basin of attraction is much larger, and, for sufficiently large $\kappa$, is practically the entire unit circle. Thus, we formulate our second conjecture as follows:
\begin{conjecture}\label{AuxConj}
There is a universal constant $C\ge 1$ such that whenever $\kappa >CD(\Omega)$, any generic initial phase vector $\Theta^0$ converges to the phase-locked state given in Theorem \ref{T2.2}.
\end{conjecture}
\noindent Again we say `generic' because of Example \ref{NoSync} and of the possible existence of unstable phase-locked states.

Asymptotic phase-locking has been well established for the case of identical oscillators, i.e., $D(\Omega) = 0$. In fact, Dong and Xue \cite{D-X} showed that asymptotic phase-locking is attained for all initial data with $D(\Theta^0) < 2\pi$. The idea of the proof is that, in the context of Theorem \ref{T2.1}, there are no collisions. Furthermore, Watanabe and Strogatz \cite{W-S} were able to construct $(N-3)$ constants of motion to transform system \eqref{Ku} into an $(N-3)$-parameter family of two-dimensional dynamics (\eqref{Ku} is itself an $(N-1)$-dimensional system if we only care about the phase differences), and were thus able to describe the detailed relaxation process as well as the emergent phase-locked state.

However, these methods fail in the case of nonidentical oscillators $D(\Omega) > 0$; it seems that in this case asymptotic phase-locking is a more subtle matter. In particular, for $N\ge 4$ there are pathological examples with $R(t)\equiv 0$ that cannot synchronize even for arbitrarily large $\kappa$:

\begin{example}\label{NoSync}
\emph{(Non-synchronizing Kuramoto flow)}
For $N\ge 4$, we have non-synchronizing choices for $\Theta^0$ and $\Omega$. For example, consider the system
\begin{align*}
\begin{aligned}
& N=4,\quad \nu_1=\nu_2\neq \nu_3=\nu_4,\cr
& 
\theta_1=\nu_1 t, \quad \theta_2=\nu_1 t+\pi, \quad 
\theta_3=\nu_3 t,\quad \theta_4=\nu_3 t+\pi.
\end{aligned}
\end{align*}
Then the pairs ($\theta_1$, $\theta_2$) and ($\theta_3$, $\theta_4$) each form a configuration with zero centroid that rotate independently of each other. The order parameter $R$ of the entire configuration is always zero:
\[  R = \frac{1}{N} (e^{{\mathrm i} \nu_1 t} +  e^{{\mathrm i} \nu_1 t + \pi {\mathrm i}} +  e^{{\mathrm i} \nu_3 t} +  e^{{\mathrm i} \nu_3 t + \pi {\mathrm i}} ) = 0, \quad t \geq 0. \]
So the nonlinear interaction term of \eqref{Ku-mf} vanishes, no matter how large the coupling $\kappa$ is; hence the above system is a solution to \eqref{Ku-mf} and thus also to \eqref{Ku}. Therefore, this solution fails to achieve asymptotic phase-locking. Similarly, we can construct analogous examples for $N\ge 5$ by superposing two or more configurations, each with zero centroid, which are rotating independently.
\end{example}

\noindent This is particularly discouraging because this precludes the possibility of constructing a nontrivial Riemannian metric on $\bbt^N$ which makes \eqref{Ku} into a gradient potential flow on $\bbt^N$. Hence for the nonidentical case $D(\Omega)>0$, in order to prove Conjecture \ref{MainConj} we would need to use new tools fundamentally different from those used in the identical case $D(\Omega)=0$. 

One promising direction in the nonidentical case seems to study the dynamics of the order parameters: recently Ha et al. \cite{H-K-P} extended the basin of attraction of the phase-locked state of Theorem \ref{T2.2}. Namely, for initial data $\Theta^0$ with positive initial order parameter $R_0>0$, with $|\theta_i^0-\phi^0|<\frac{\pi}{2}+\varepsilon$ where $\varepsilon>0$ is sufficiently small, and with sufficiently large coupling $\kappa \gg D(\Omega)$,  asymptotic phase-locking towards the phase-locked state of Theorem \ref{T2.2} will occur. Although it was not able to treat generic initial data $\Theta^0$, it was the first, to the authors' knowledge, to extend the basin of attraction beyond configurations confined in a half-circle.

Asymptotic phase-locking for generic initial data was first obtained by Ha et al. \cite{H-K-R} in a sufficiently large coupling regime. An important piece of the argument of \cite{H-K-R} is that some statements of Theorem \ref{T2.2} can be extended to the case where we have only a majority, not the totality, of the population lying on a small arc:
\begin{proposition}{\cite{H-K-R}}\label{P2.2}
Suppose that the initial configuration $\Theta^0$ and parameters $n_0\in \bbn$, and $\ell, \kappa \in \bbr$ satisfy
\begin{align*}
\begin{aligned}
&\frac{1}{2} < \gamma_0 := \frac{n_0}{N} \leq 1, \quad \ell \in\left(0,2\cos^{-1} \Big( \frac{1}{\gamma_0} - 1 \Big) \right), \quad \theta^0_{j}\in[-\pi,\pi),\quad 1\leq j\leq N, \\
& \max_{1\leq j,k\leq n_0}|\theta_{j0}-\theta_{k0}|< \ell, \quad  \kappa > \frac{D(\Omega)}{\gamma_0 \sin \ell -2 (1-\gamma_0) \sin \frac{\ell}{2}}.
\end{aligned}
\end{align*}
Then, for any solution $\Theta$ to \eqref{Ku}, we have
\[
\sup_{0\leq t<\infty}D(\Theta(t))\leq 4\pi+ \ell \quad \mbox{and} \quad \lim_{t\rightarrow\infty}\|\dot{\Theta}(t)\|_\infty=0.
\]
\end{proposition}
\noindent In words, this means that if we have a majority $\gamma>1/2$ of the oscillators lying on a small arc of length $\ell<2\cos^{-1}\left(\frac{1}{\gamma}-1\right)$, then with sufficiently large coupling $\kappa$, the ensuing Kuramoto flow will exhibit asymptotic phase-locking. With Proposition \ref{P2.2} at hand, the authors in \cite{H-K-R} were able to show the emergence of complete synchronization for a generic initial phase configuration, summarized as follows.
\begin{theorem}\label{T2.3}
\emph{\cite{H-K-R}}
Suppose that the initial configuration $\Theta^0$ satisfies the conditions:
\begin{equation*}\label{generic}
R_0 >0 \quad \mbox{and} \quad \theta_i^0\not\equiv \theta_j^0 \mod 2\pi, \quad \forall~i \neq j.
\end{equation*}
 Then for sufficiently large coupling strength $\kappa$, the solution $\Theta(t)$ to \eqref{Ku} with initial data $\Theta^0$  converges to a  phase-locked state $\Theta^{\infty}$:
\[ \lim_{t \to \infty} ||\Theta(t) - \Theta^{\infty}||_{\infty} = 0. \]
\end{theorem}
\noindent The idea is to use the comparison principle between the identical case and the nonidentical case. Since we know that asymptotic phase-locking happens in the identical case, we can enlarge $\kappa$ to restrict the behavior of nonidentical oscillators close to a phase-locked state of the identical case, and then apply Proposition \ref{NoChaos} and Proposition \ref{P2.2}. Although the paper \cite{H-K-R} provided the resolution of asymptotic phase-locking for generic initial phase vectors in a large coupling strength regime, it was not able to provide an upper bound for the required coupling strength $\kappa$ independent of $N$. One of main contributions of this paper is to provide an explicit $N$-independent upper bound on $\kappa$ that gives rise to asymptotic phase-locking (see Theorem \ref{T3.2} in the next section). 
\section{Description of main results} \label{sec:3}
\setcounter{equation}{0} 
In this section, we briefly summarize our three main results, namely emergence of partial phase-locking and two sufficient coupling strengths for complete phase-locking. Each subsection will be devoted to one main result, along with a comparison with earlier known results.
\subsection{Partial phase-locking of majority ensembles} \label{sec:3.1}
In this subsection, we briefly present our first main result on the asymptotic emergence of partial phase-locking. Our first result generalizes the complete phase-locking result of Theorem \ref{T2.2} to partial phase-locking.

Recall that the result of Theorem \ref{T2.2} says that if we have all of the oscillators initially lying on some proper subarc of the half-circle, then, with sufficiently large $\kappa$, we can predict the behavior of the ensuing Kuramoto flow.  We will show that a similar result holds even if we know only a majority (i.e., more than half) of the oscillators lie on some small arc. The reason for this is that if we have sufficient control over a majority of the oscillators, then we have enough data on the mean field term of \eqref{Ku}. Before we present our first result, we introduce some related terminology.
\begin{definition} \label{D3.1}
Let $\mathcal{A}\subset\mathcal{N}$ be a collection of indices, let $\gamma\in (0.5,1]$, $t_0\ge 0$ and $\ell>0$ be real numbers, and let $\Theta(t)$ be a solution to \eqref{Ku}.
\begin{enumerate}
\item The ensemble $\Theta_\mathcal{A}=(\theta_i)_{i\in \mathcal{A}}$ is a $\gamma$-ensemble if $\frac{|\mathcal{A}|}{N}\ge \gamma$.
\item The ensemble $\Theta_\mathcal{A}(t_0)$ is a $\gamma$-ensemble of arclength $\le \ell$ at time $t_0$ if it is a $\gamma$-ensemble and if, after some suitable modulo $2\pi$ shifts, we have $D(\Theta_\mathcal{A}(t_0))\le \ell$.
\item The ensemble $\Theta_\mathcal{A}$ is a ``stable" $\gamma$-ensemble of arclength $\le \ell$ at time $t_0$ if it is a $\gamma$-ensemble and if, after some suitable modulo $2\pi$ shifts, we have $D(\Theta_\mathcal{A}(t))\le \ell$ for all $t\ge t_0$.
\end{enumerate}
\end{definition}
\begin{remark} \label{R3.1}
If there is no confusion, we will sometime say $\mathcal{A}$ itself is a (stable) $\gamma$-ensemble (of arclength $\le \ell$).
\end{remark}
We are now ready to present our first main results on the formation of stable $\gamma$-ensembles and partial phase-locking ensembles. For the parameters $\gamma,\ell,\kappa$ and $\Omega_{\mathcal B}$, we set
\[ \kappa_*(\gamma, \ell, D(\Omega_\mathcal{B})) := \frac{D(\Omega_\mathcal{B})}{\gamma \sin \ell -2 (1-\gamma) \sin \frac{\ell}{2}}. \]
\begin{theorem}\label{T3.1}
For given sets of indices $\mathcal{A}\subset\mathcal{B}\subset\mathcal{N}$, suppose that the parameters $\gamma,\ell$ and $\kappa$ satisfy
\begin{equation} \label{B-0}
\frac{1}{2} < \gamma \leq 1, \quad \ell \in\left(0,2\cos^{-1} \Big( \frac{1}{\gamma} - 1 \Big) \right), \quad \kappa > \kappa_*(\gamma, \ell, D(\Omega_\mathcal{B})),
\end{equation}
and let $\Theta$ be a solution to \eqref{Ku} such that $\Theta_\mathcal{A}$ is a $\gamma$-ensemble of arclength $\le \ell$ at time $0$. Then the following assertions hold.
\begin{enumerate}
\item The ensemble $\Theta_\mathcal{A}$ is stable:
\[
\sup_{0\leq t<\infty}D(\Theta_\mathcal{A}(t))\leq \ell\quad \mbox{and}\quad\limsup_{t\rightarrow\infty}D(\Theta_\mathcal{A}(t))\leq \phi_1(\gamma, \kappa, D(\Omega_\mathcal{B})),
\]
where $\phi_1$ is the smaller root of the following trigonometric equation:
\[  \gamma\sin\theta-2(1-\gamma)\sin\frac{\theta}{2} = \frac{D(\Omega_\mathcal{B})}{\kappa} \quad \mbox{in}~~\Big(0,2\cos^{-1}\frac{1-\gamma}{\gamma} \Big). \]
\item The ensemble $\Theta_\mathcal{B}$ is partially phase-locked:
\[
\sup_{0\leq t<\infty}D(\Theta_\mathcal{B}(t))<\infty.
\]
\item If we assume in addition that
\begin{equation}\label{additional0}
\frac{D(\Omega_\mathcal{B})}{\kappa}< \frac{(2\gamma-1)^{3/2}}{\sqrt{2\gamma}}\frac{2-\gamma}{\sqrt{\gamma/2}+(1-\gamma)},
\end{equation}
then the oscillators of $\Theta_\mathcal{A}$ becomes ordered according to their natural frequencies: for $i,j\in \mathcal{A}$, \newline
\begin{enumerate}
\item
if $\nu_i>\nu_j$, then
\begin{align*}
\begin{aligned}
& \frac{\nu_i-\nu_j}{\kappa}\le \liminf_{t\rightarrow\infty}[\theta_i(t)-\theta_j(t)]\le \limsup_{t\rightarrow\infty}[\theta_i(t)-\theta_j(t)] \le \frac{\pi}{2\sqrt{2}(\gamma\cos\phi_1-(1-\gamma))}\frac{\nu_i-\nu_j}{\kappa}, 
\end{aligned}
\end{align*}
\item
and if $\nu_i=\nu_j$, then
\[ \lim_{t\rightarrow\infty}[\theta_i(t)-\theta_j(t)]=0. \]
\end{enumerate}
\end{enumerate}
\end{theorem}
In words, our first main result says that a $\gamma$-ensemble of arclength $\le\ell<2\cos^{-1} \Big( \frac{1}{\gamma_0} - 1 \Big)$ is stable for sufficiently strong coupling strength $\kappa$, and that the diameter of that ensemble will have a limit supremum bounded above by $\phi_1$. By Lemma \ref{f-estimates} (1), we have
\[ \phi_1< \frac{3\pi}{4(2\gamma-1)}\frac{D(\Omega_\mathcal{B})}{\kappa}, \]
so we may make the limiting diameter as small as we wish by increasing $\kappa$. Moreover, we can control a bigger part $\mathcal{B}$ of the population, simply by making $\kappa$ large enough to satisfy $\eqref{B-0}_3$. In the extremal case $\mathcal{B}=\mathcal{N}$, the total diameter $D(\Theta(t))$ will stay bounded, and hence by Proposition \ref{NoChaos} we will have asymptotic phase-locking. This is the result we will use in Section \ref{sec:5}, and so we record this separately in the following corollary.
\begin{corollary}\label{C3.1}
Suppose that $\Theta$ is a solution to \eqref{Ku}, and that $\Theta_\mathcal{A}$ is a $\gamma$-ensemble of arclength $\le \ell$ at time $0$, where the parameters $\gamma,l, \kappa \in \bbr$ satisfy
\begin{equation} \label{B-1}
\frac{1}{2} < \gamma \leq 1, \quad \ell \in\left(0,2\cos^{-1} \Big( \frac{1}{\gamma} - 1 \Big) \right), \quad \kappa> \kappa_*(\gamma, \ell, D(\Omega)).
\end{equation}
Then the following assertions hold. 
\begin{enumerate}
\item The ensemble $\Theta_\mathcal{A}$ is stable:
\begin{equation}\label{old0}
\sup_{0\leq t<\infty}D(\Theta_\mathcal{A}(t))\leq \ell
\end{equation}
and
\begin{equation}\label{notold}
\limsup_{t\rightarrow\infty}D(\Theta_\mathcal{A}(t))\leq \phi_1(\gamma, \kappa, D(\Omega))\le \frac{3\pi}{4(2\gamma-1)}\frac{D(\Omega)}{\kappa}.
\end{equation}
\item Asymptotic phase-locking occurs:
\begin{equation}\label{old1}
\sup_{0\leq t<\infty}D(\Theta(t))<\infty.
\end{equation}
\item If we assume in addition that
\begin{equation}\label{additional}
\frac{D(\Omega)}{\kappa}< \frac{(2\gamma-1)^{3/2}}{\sqrt{2\gamma}}\frac{2-\gamma}{\sqrt{\gamma/2}+(1-\gamma)}
\end{equation}
then the oscillators of $\Theta_\mathcal{A}$ become ordered according to their natural frequencies: for $i,j\in \mathcal{A}$ with $\nu_i\ge \nu_j$ we have
\[
\frac{\nu_i-\nu_j}{\kappa}\le \liminf_{t\rightarrow\infty}[\theta_i(t)-\theta_j(t)]\le \limsup_{t\rightarrow\infty}[\theta_i(t)-\theta_j(t)]
\le \frac{\pi}{2\sqrt{2}(\gamma\cos\phi_1-(1-\gamma))}\frac{\nu_i-\nu_j}{\kappa}.
\]
\end{enumerate}
\end{corollary}
\begin{remark}\label{novelty}
The results \eqref{old0} and \eqref{old1} under conditions \eqref{B-1} have been previously observed in \cite{H-K-R}, namely as Proposition \ref{P2.2}. The novelty of our results compared to the results in \cite{H-K-R} is that
\begin{itemize}
\item[(a)] we generalize the results to partial phase-locking,
\item[(b)] we obtain a stronger upper bound on the asymptotic phase diameter, namely \eqref{notold},
\item[(c)] and we describe the limiting behavior of the stable ensemble according to the corresponding natural frequencies, namely (3) above.
\end{itemize}
\end{remark}
\subsection{Complete phase-locking I (pathwise, $R_0$-dependent)} \label{sec:3.2}
In this subsection, we state our second main result for complete phase-locking which improves the result of Theorem \ref{T2.3} by replacing the $N$-dependent lower bound for $\kappa$ by an $N$-independent lower bound. More precisely, our second main result can be stated as follows.
\begin{theorem}\label{T3.2}
Suppose that $\Theta^0$, $\Omega$ and $\kappa$ satisfy
\[
R_0 :=  R(\Theta^0) > 0, \quad \kappa >1.6\frac{D(\Omega)}{R_0^2},
\]
and let $\Theta = \Theta(t)$ be the solution to \eqref{Ku} with initial data $\Theta^0$. Then, the following assertions hold.
\begin{enumerate}
\item Asymptotic complete phase-locking occurs.

\vspace{0.2cm}

\item Let
\[
\gamma(R_0)=
\begin{cases}
0.5+\frac{0.35}{0.94}R_0& 0<R_0\le 0.94, \\
1-2.5(1-R_0)& 0.94<R_0\le 1.
\end{cases}
\]
Then $\Theta(t)$ has a stable $\gamma(R_0)$-ensemble of arclength $\le \frac{3\pi}{4(2\gamma(R_0)-1)}\frac{D(\Omega)}{\kappa}$.

\vspace{0.2cm}

\item The aforementioned stable $\gamma(R_0)$-ensemble becomes ordered in accordance with its natural frequencies: if $\theta_i$ and $\theta_j$ belong to that ensemble and $\nu_i\ge \nu_j$, then
\[
\frac{\nu_i-\nu_j}{\kappa}\le \liminf_{t\rightarrow\infty}[\theta_i(t)-\theta_j(t)]\le \limsup_{t\rightarrow\infty}[\theta_i(t)-\theta_j(t)]
\le c\frac{\nu_i-\nu_j}{\kappa},
\]
where the constant $c$ is defined as
\[
c :=\frac{\pi}{2\sqrt{2}(\gamma\cos\phi_1(\gamma, \kappa, D(\Omega))-(1-\gamma))}.
\]
\end{enumerate} 
\end{theorem}
\begin{remark} 
1. The lower bound 
\[  \kappa_*(\gamma_0, \ell, D(\Omega)) = \frac{D(\Omega)}{\gamma_0 \sin \ell -2 (1-\gamma_0) \sin \frac{\ell}{2}} \]
for $\kappa$ in Proposition \ref{P2.2} will contribute to the lower bound $\frac{1.6}{R_0^2} D(\Omega)$ of Theorem \ref{T3.2} by making a suitable $R_0$-dependent choice of $\gamma$ and $\ell$. See Lemma \ref{L5.3} for our specific choice. \newline

\noindent 2. As a direct corollary, we have an upper bound for the pathwise critical coupling strength:
\[ \kappa_{pc}(\Theta^0,\Omega, N)\le \frac{1.6}{R_0^2} D(\Omega). \]

\noindent 3. The authors have previously proven a weaker version of Theorem \ref{T3.2} in \cite{H-Ko-R} for a related synchronization model on spheres of arbitrary dimension, namely that it is possible to concentrate a majority of the oscillators on the sphere to an arbitrarily small open ball, at the cost of increasing the coupling strength to the order of $R_0^{-2}$. Part of the proof of Theorem \ref{T3.2} draws some computational techniques from \cite{H-Ko-R}. The choice of parameters (specified in Lemma \ref{L5.3}) are more optimal compared to \cite{H-Ko-R} and thus the required coupling strength is smaller.
\end{remark}

\subsection{Complete phase-locking II(pathwise, uniform, $N$-dependent)} \label{sec:3.3}
In this subsection, we present our third main result on the pathwise critical coupling strength $\kappa_{pc}(\Omega, N)$. Ultimately, we would like to attain the bound of Conjecture \ref{MainConj}, and get rid of the $\frac{1}{R_0^2}$ factor of Theorem \ref{T3.2}. This was verified for $N=2$ in Example \ref{N=2}, and we will achieve this for $N= 3$ (see Corollary \ref{C3.3} and Proposition \ref{P4.1}). However, for $N\ge 4$, the situation becomes a little bit complicated: Example \ref{NoSync} gives us pathological examples with $R(t)\equiv 0$ that cannot synchronize even for arbitrarily large $\kappa$. Actually, failure to exhibit asymptotic phase-locking and having a low order parameter $R$ for all times is a closely related property: from the fact that system \eqref{Ku} is autonomous, we easily deduce from Theorem \ref{T3.2} the following corollary.
\begin{corollary} \label{C3.2}
Let $\Theta:\bbr\rightarrow \bbr^N$ be a solution to \eqref{Ku} on $\bbr^N$. If $\Theta$ does not exhibit asymptotic phase-locking, then we must necessarily have
\[
R(t)\le \sqrt{\frac{1.6}{\kappa} D(\Omega)},\quad t\in \bbr.
\]
\end{corollary}
While we cannot completely rule out the possibility that $\kappa_{pc}(\Theta^0,\Omega, N)$ has a singularity at $R_0=0$, we can show that for fixed finite $N$, Conjecture \ref{MainConj} holds with $C=1.6N$. This is because solutions to \eqref{Ku} with extremely low order parameters for all times are inherently unstable.

\begin{theorem} \label{T3.3}
Suppose that for fixed $N$, $\Omega=(\nu_1,\cdots,\nu_N)$, the coupling strength satisfies
\[
\kappa>1.6ND(\Omega).
\]
Then for almost all initial data $\Theta^0$, the following statements are true for the ensuing Kuramoto flow $\Theta(t)$ for \eqref{Ku}.
\begin{enumerate}
\item Complete phase-locking occurs.

\vspace{0.2cm}

\item Let
\[
\gamma_N=0.5+\frac{0.35}{0.94\sqrt{N}}.
\]
Then, $\Theta(t)$ has a stable $\gamma_N$-ensemble of arclength $\le \frac{3\pi}{4(2\gamma_N-1)}\frac{D(\Omega)}{\kappa}$.

\vspace{0.2cm}

\item The aforementioned stable $\gamma_N$-ensemble becomes ordered in accordance with its natural frequencies: if $\theta_i$ and $\theta_j$ belong to that ensemble and $\nu_i\ge \nu_j$, then
\[
\frac{\nu_i-\nu_j}{\kappa}\le \liminf_{t\rightarrow\infty}[\theta_i(t)-\theta_j(t)]\le \limsup_{t\rightarrow\infty}[\theta_i(t)-\theta_j(t)]
\le c\frac{\nu_i-\nu_j}{\kappa},
\]
where the constant $c$ is defined as
\[
c:=\frac{\pi}{2\sqrt{2}(\gamma_N\cos\phi_1(\gamma_N, \kappa, D(\Omega))-(1-\gamma_N))}.
\]
\end{enumerate}
\end{theorem}

As a direct application of Theorem \ref{T3.3}, we also have the following results.
\begin{corollary}\label{C3.3} The following assertions hold.
\begin{enumerate}
\item
Let $N$ and $\Omega=(\nu_1,\cdots,\nu_N)$ be fixed. Then, we have
\[
\kappa_{pc}(\Omega, N)\le 1.6ND(\Omega)\quad \mbox{ for a.e. } \Theta^0.
\]
\item
Let $N=3,~\Omega=(\nu_1,\nu_2, \nu_3)$ and $\kappa>4.8D(\Omega)$. Then for almost all initial data $\Theta^0$, the solution $\Theta(t)$ to \eqref{Ku} converges to the phase-locked state of Theorem \ref{T3.2}.
\end{enumerate}
\end{corollary}
\begin{proof}
(i)~The first assertion follows directly from Theorem \ref{T3.3}. \newline

\noindent (ii)~The above theorem tells us that there will be a stable $\gamma_3$-ensemble $\mathcal{A}$ of arclength $\le \ell$, where
\[
\gamma_3=0.5+\frac{0.35}{0.94\sqrt{3}}=0.71497..>\frac{2}{3}
\]
and
\[
\ell=\frac{3\pi}{4(2\gamma_3-1)}\frac{D(\Omega)}{\kappa}<\frac{3\pi}{4(2\gamma_3-1)}\cdot \frac{1}{4.8}=0.363..\cdot \pi<\frac{\pi}{2}.
\]
Hence $\mathcal{A}$ is actually a $1$-ensemble, i.e., $\mathcal{A}=\mathcal{N}$, of arclength $\le \frac{\pi}{2}$, and so Theorem \ref{T2.2} applies at some finite time. 

\end{proof}

\section{Partial phase-locking of majority ensembles}\label{sec:4}
\setcounter{equation}{0}
In this section, we complete the discussion of subsection \ref{sec:3.1} by providing a proof of Theorem \ref{T3.1}. First, we assume that $\mathcal{A}\subset \mathcal{B}\subset\mathcal{N}$ are collections of indices and that $\gamma\in (0.5,1]$ is a real number such that $\mathcal{A}$ is a $\gamma$-ensemble (majority ensemble). Under some conditions, we will show that $\Theta_\mathcal{A}$ forms a stable $\gamma$-ensemble and that $\mathcal{B}$ is a partial phase-locking ensemble in the sense of Definition \ref{D2.3}.

\subsection{Derivation of Gronwall's inequality for $D(\Theta_\mathcal{A})$} In this subsection, we will derive a nonlinear Gronwall inequality for the diameter $D(\Theta_\mathcal{A})$.
\begin{lemma}\label{gronwall0}
Let $\mathcal{A}\subset \mathcal{B}\subset\mathcal{N}$ be collections of indices such that $\mathcal{A}$ is a $\gamma$-ensemble and $D(\Theta_\mathcal{A}(t_0))<2\pi$. Then, we have
\begin{equation}\label{gronwall}
\left.\frac{d}{dt}\right|_{t = t_0}^+D(\Theta_\mathcal{A})\le D(\Omega_\mathcal{B})-2\kappa \Big( \sin\frac{D(\Theta_\mathcal{A})}{2} \Big) \left(\gamma \cos\frac{D(\Theta_\mathcal{A})}{2}-(1-\gamma)\right),
\end{equation}
where $\left.\frac{d}{dt}\right|^+$ is the upper Dini derivative.
\end{lemma}
\begin{proof} We choose indices $i$ and $j$ such that
\[
\theta_i(t_0)=\min_{k\in\mathcal{A}} \theta_k(t_0),\quad \theta_j(t_0)=\max_{k\in\mathcal{A}} \theta_k(t_0).
\]
Then, for such $i$ and $j$, at time $t=t_0$, we use elementary trigonometry identities to get 
\begin{align*}
\dot{\theta}_j-\dot{\theta}_i&=\nu_j-\nu_i+\frac{\kappa}{N}\sum_{k=1}^N\left[\sin(\theta_k-\theta_j)-\sin(\theta_k-\theta_i)\right]\\
&\le D(\Omega_\mathcal{B})+\frac{2\kappa}{N}\sum_{k=1}^N \cos\left(\theta_k-\frac{\theta_i+\theta_j}{2}\right)\sin\frac{\theta_i-\theta_j}{2}\\
&= D(\Omega_\mathcal{B})-\frac{2\kappa}{N}\sin\frac{\theta_j-\theta_i}{2}\sum_{k=1}^N \cos\left(\theta_k-\frac{\theta_i+\theta_j}{2}\right)\\
&= D(\Omega_\mathcal{B})-2\kappa \sin\frac{D(\Theta_\mathcal{A})}{2}\left[\frac{1}{N}\sum_{k\in\mathcal{A}} \cos\left(\theta_k-\frac{\theta_i+\theta_j}{2}\right)+\frac{1}{N}\sum_{k\in\mathcal{N}\backslash\mathcal{A}} \cos\left(\theta_k-\frac{\theta_i+\theta_j}{2}\right)\right]\\
&\le D(\Omega_\mathcal{B})-2\kappa \sin\frac{D(\Theta_\mathcal{A})}{2}\left[\frac{|\mathcal{A}|}{N}\cos\frac{D(\Theta_\mathcal{A})}{2}-\frac{N-|\mathcal{A}|}{N}\right]\\
&\le D(\Omega_\mathcal{B})-2\kappa \sin\frac{D(\Theta_\mathcal{A})}{2}\left[\gamma\cos\frac{D(\Theta_\mathcal{A})}{2}-(1-\gamma)\right],
\end{align*}
where in the penultimate inequality we used the relation
\[ \Big| \theta_k-\frac{\theta_i+\theta_j}{2} \Big| = \Big |\frac{\theta_k - \theta_i + \theta_k - \theta_j}{2} \Big| \leq \frac{|\theta_k - \theta_i | + |\theta_k - \theta_j|}{2} \leq \frac{D(\Theta_{\mathcal A})}{2},     \]
which follows from the choice of $i$ and $j$.
\end{proof}
\begin{remark}
Note that if $\gamma\le \frac{1}{2}$, then the mean-field term of \eqref{gronwall} would be nonpositive:
\[ \left(\gamma \cos\frac{D(\Theta_\mathcal{A})}{2}-(1-\gamma)\right) \leq 2 \gamma - 1<0, \]
 and hence Lemma \ref{gronwall0} would be practically useless. On the other hand, if $\gamma> \frac{1}{2}$, then the mean-field can be made positive for sufficiently small $D(\Theta_\mathcal{A})$. That is why we assume $\gamma> \frac{1}{2}$, i.e., the case of majority ensembles.
\end{remark}

\noindent Next, we estimate the behavior of the coupling term appearing in the right-hand side of the nonlinear Gronwall inequality \eqref{gronwall}. To this end, consider the function $f:\bbr\rightarrow\bbr$ given by
\begin{equation} \label{NNN-1}
f(\theta) =2\sin\frac{\theta}{2}\left(\gamma\cos\frac{\theta}{2}-(1-\gamma)\right)=\gamma\sin\theta-2(1-\gamma)\sin\frac{\theta}{2},
\end{equation}
so that \eqref{gronwall} becomes
\begin{equation}\label{gronwall2}
\left.\frac{d}{dt}\right|_{t_0}^+D(\Theta_\mathcal{A})\le D(\Omega_\mathcal{B})- \kappa f(D(\Theta_\mathcal{A})) =  \kappa \Big[ \frac{D(\Omega_\mathcal{B})}{\kappa} -    f(D(\Theta_\mathcal{A})) \Big ].
\end{equation}
We list some properties of $f$ in the following lemma.
\begin{lemma} \label{L4.2}
Let $f$ be the function defined in \eqref{NNN-1}. Then we have the following assertions.
\begin{enumerate}
\item
The function $f$ has zeros at $\theta=0$ and $\theta=2\cos^{-1} \Big( \frac{1-\gamma}{\gamma} \Big)$, and $f$ is positive on the interval $(0,2\cos^{-1}\frac{1-\gamma}{\gamma})$.
\item
On the interval $(0,2\cos^{-1}\frac{1-\gamma}{\gamma})$, the function $f$ is strictly concave and attains its maximum at the unique zero $\theta_*=2\cos^{-1} \Big( \frac{1-\gamma+\sqrt{(1-\gamma)^2+8\gamma^2}}{4\gamma} \Big)$ of $f'(\theta_*)=0$ in $(0,2\cos^{-1}\frac{1-\gamma}{\gamma})$. 
\end{enumerate}
\end{lemma}
\begin{proof}
(i)~It follows from \eqref{NNN-1} that 
\[ f(\theta) = 0 \quad \Longleftrightarrow \quad \sin\frac{\theta}{2} = 0 \quad \mbox{or} \quad \gamma\cos\frac{\theta}{2}-(1-\gamma) = 0. \]
This yields the desired roots of $f$. It is easy to see that $f$ is positive on the interval $(0,2\cos^{-1}\frac{1-\gamma}{\gamma})$. \newline

\noindent (ii)~By direct calculation, we have
\begin{align*}
f''(\theta)&=-\gamma\sin\theta+\frac{1-\gamma}{2}\sin\frac{\theta}{2}=\sin\frac{\theta}{2}\left(-2\gamma\cos\frac{\theta}{2}+\frac{1-\gamma}{2}\right)\\
&<\sin\frac{\theta}{2}\left(-2(1-\gamma)+\frac{1-\gamma}{2}\right)=-\frac{3(1-\gamma)}{2}\sin\frac{\theta}{2}<0,\quad \theta\in \Big(0,2\cos^{-1}\frac{1-\gamma}{\gamma} \Big).
\end{align*}
Thus $f$ is strictly concave on $(0,2\cos^{-1}\frac{1-\gamma}{\gamma})$, and has a maximum at the unique zero $\theta_*=\theta_*(\gamma)$ of $f'(\theta_*)=0$ in $(0,2\cos^{-1}\frac{1-\gamma}{\gamma})$. We calculate $\theta_*$ as follows:
\begin{align}\label{max}
\begin{aligned}
f'(\theta_*)=0 &\quad\Longleftrightarrow\quad \gamma\cos\theta_*-(1-\gamma)\cos\frac{\theta_*}{2}=0\\
&\quad \Longleftrightarrow\quad 2\gamma \cos^2\frac{\theta_*}{2}-(1-\gamma)\cos\frac{\theta_*}{2}-\gamma=0\\
&\quad\Longleftrightarrow\quad \cos\frac{\theta_*}{2}=\frac{1-\gamma+\sqrt{(1-\gamma)^2+8\gamma^2}}{4\gamma} \qquad (\because \cos\frac{\theta_*}{2}>0)\\
&\quad\Longleftrightarrow\quad \theta_*=2\cos^{-1} \Big( \frac{1-\gamma+\sqrt{(1-\gamma)^2+8\gamma^2}}{4\gamma} \Big).
\end{aligned}
\end{align}
\end{proof}

\bigskip

Note that if $\frac{D(\Omega_\mathcal{B})}{\kappa}<f(\theta_*)$, then $f(\theta)=\frac{D(\Omega_\mathcal{B})}{\kappa}$ will have two zeros $\phi_1=\phi_1(\gamma, \kappa,D(\Omega_\mathcal{B}))$ and $\phi_2=\phi_2(\gamma, \kappa,D(\Omega_\mathcal{B}))$ in $(0,2\cos^{-1}\frac{1-\gamma}{\gamma})$, with the ordering
\[
0<\phi_1<\theta_*<\phi_2<2\cos^{-1} \Big( \frac{1-\gamma}{\gamma} \Big).
\]

Next, we collect some facts which will be useful later.
\begin{lemma}\label{f-estimates}
Let $\theta_*, \phi_1$ and $\phi_2$ be defined as above. Then the following estimates hold.
\begin{eqnarray*}
&& (i)~\phi_1< \frac{3\pi}{4(2\gamma-1)}\frac{D(\Omega_\mathcal{B})}{\kappa}, \quad  \cos^{-1} \Big( \frac{1-\gamma}{\gamma} \Big) \le\theta_*, \\
&& (ii)~f\Big(\cos^{-1}\frac{1-\gamma}{\gamma} \Big)=\frac{(2\gamma-1)^{3/2}}{\sqrt{2\gamma}}\frac{2-\gamma}{\sqrt{\gamma/2}+(1-\gamma)}.
\end{eqnarray*}
\end{lemma}
\begin{proof}
(i)~For the first estimate, we use the strict concavity of $f$ to get
\begin{equation}\label{f-concav}
f(\theta)> \frac{f(\theta_*)}{\theta_*}\theta,\quad \theta\in(0, \theta_*).
\end{equation}
On the other hand, we have
\begin{align*}
\begin{aligned}
\frac{f(\theta_*)}{\theta_*}&=\frac{\sin\frac{\theta_*}{2}}{\theta_*/2}\left(\gamma\cos\frac{\theta_*}{2}-(1-\gamma)\right) \stackrel{\eqref{max}}{=}\frac{\sin\frac{\theta_*}{2}}{\theta_*/2}\frac{\sqrt{(1-\gamma)^2+8\gamma^2}-3(1-\gamma)}{4}\\
&=\frac{\sin\frac{\theta_*}{2}}{\theta_*/2}\frac{2(2\gamma-1)}{\sqrt{(1-\gamma)^2+8\gamma^2}+3(1-\gamma)} >\frac{\sin\frac{\theta_*}{2}}{\theta_*/2}\frac{2(2\gamma-1)}{\sqrt{\gamma^2+8\gamma^2}+3(1-\gamma)}~~(\because 0\le 1-\gamma<\gamma)\\
&=\frac{\sin\frac{\theta_*}{2}}{\theta_*/2}\frac{2(2\gamma-1)}{3} \ge\frac{2}{\pi}\frac{2(2\gamma-1)}{3} \quad (\because x\mapsto \frac{\sin x}{x} \mbox{ is decreasing on } \Big(0,\frac{\pi}{2} \Big)\mbox{ and }\theta_*\in(0,\pi))\\
&=\frac{4(2\gamma-1)}{3\pi}>0.
\end{aligned}
\end{align*}
Note that all terms involving $\gamma$ in the above calculation are positive because $\gamma> \frac{1}{2}$. Now we substitute $\theta=\phi_1$ into \eqref{f-concav} to obtain
\[
\frac{D(\Omega_\mathcal{B})}{\kappa}=f(\phi_1)> \frac{f(\theta_*)}{\theta_*}\phi_1> \frac{4(2\gamma-1)}{3\pi}\phi_1.
\]
This yields
\[
\phi_1< \frac{3\pi}{4(2\gamma-1)}\frac{D(\Omega_\mathcal{B})}{\kappa}. \]

\noindent For the second estimate, note that 
\begin{align*}
\cos^{-1}\Big( \frac{1-\gamma}{\gamma} \Big) \le\theta_* & \quad\Longleftrightarrow\quad \frac{1-\gamma}{\gamma}\ge\cos\theta_*\\
&\quad\Longleftrightarrow\quad \frac{1-\gamma}{\gamma}\ge\frac{1-\gamma}{\gamma}\cos\frac{\theta_*}{2}~~(\because \mbox{first line of }\eqref{max})\\
&\quad\Longleftarrow\quad 1\ge \cos\frac{\theta_*}{2}.
\end{align*}

\noindent (ii)~We set
\[ \psi :=\cos^{-1} \Big( \frac{1-\gamma}{\gamma} \Big), \quad \mbox{i.e.,} \quad \cos \psi = \frac{1-\gamma}{\gamma}. \]
We can calculate
\[
\cos\frac{\psi}{2}=\sqrt{\frac{1+\cos\psi}{2}}=\frac{1}{\sqrt{2\gamma}},\qquad 
\sin\frac{\psi}{2}=\sqrt{1-\cos^2\frac{\psi}{2}}=\sqrt{\frac{2\gamma-1}{2\gamma}},
\]
and thus
\begin{align*}
\begin{aligned}
f(\psi) &=2\sin\frac{\psi}{2}\left(\gamma\cos\frac{\psi}{2}-(1-\gamma)\right)=2\sqrt{\frac{2\gamma-1}{2\gamma}}\left(\sqrt{\gamma/2}-(1-\gamma)\right)\\
&=\sqrt{\frac{2\gamma-1}{2\gamma}}\frac{(2-\gamma)(2\gamma-1)}{\sqrt{\gamma/2}+(1-\gamma)}=\frac{(2\gamma-1)^{3/2}}{\sqrt{2\gamma}}\frac{2-\gamma}{\sqrt{\gamma/2}+(1-\gamma)}.
\end{aligned}
\end{align*}
\end{proof}

\subsection{Proof of Theorem \ref{T3.1}} Now, we are ready to present our proof of the result Theorem \ref{T3.1} on the formation of stable $\gamma$-ensembles and partial phase-locking.\newline

\noindent (1) (The ensemble $\Theta_{{\mathcal A}}$ is stable): Condition \eqref{B-0} tells us that
\[
f(\ell)>\frac{D(\Omega_\mathcal{B})}{\kappa},
\]
and since $f$ has its maximum in $[0,2\cos^{-1}\frac{1-\gamma}{\gamma}]$ at $\theta_*$, we also have 
\[ f(\theta_*)>\frac{D(\Omega_\mathcal{B})}{\kappa}. \] 
The two zeros $\phi_1$ and $\phi_2$ of $f(\theta)=\frac{D(\Omega_\mathcal{B})}{K}$ exist, and by the concavity of $f$ we should have the ordering
\[
0<\phi_1<\ell<\phi_2<2\cos^{-1} \Big( \frac{1-\gamma}{\gamma} \Big).
\]
Let $\psi_1,\psi_2\in\bbr$ be any reals with
\[
\phi_1<\psi_1<\ell<\psi_2<\phi_2.
\]
Then there exists a positive constant $c>0$ such that
\[
f(\theta)>\frac{D(\Omega_\mathcal{B})}{\kappa}+c,\quad \theta\in [\psi_1,\psi_2].
\]
On the other hand, it follows from Lemma \ref{gronwall0} and \eqref{gronwall2} that 
\[
\left.\frac{d}{dt}\right|_{t_0}^+D(\Theta_\mathcal{A})\le D(\Omega_\mathcal{B})- \kappa f(D(\Theta_\mathcal{A}))<D(\Omega_\mathcal{B})- \kappa \left[\frac{D(\Omega_\mathcal{B})}{\kappa}+c\right]=-\kappa c,
\]
for all $t_0$ with $D(\Theta_\mathcal{A}(t_0))\in[\psi_1,\psi_2]$. By a standard exit-time argument, we can easily establish that
\[
D(\Theta_\mathcal{A})\le \ell,\quad \forall t\ge 0,
\]
and that there is a finite time $T\ge 0$  such that
\[
D(\Theta_\mathcal{A})\le\psi_1,\quad \forall t\ge T.
\]
Since $\psi_1\in (\phi_1,\ell)$ was arbitrary, we conclude that
\[
\limsup_{t\rightarrow\infty}D(\Theta_\mathcal{A})\le \phi_1.
\]
\vspace{0.5cm}

\noindent (2) (The ensemble $\Theta_{{\mathcal B}}$ is partially phase-locked): We set
\[
\theta_{min}=\min_{k\in\mathcal{A}}\theta_k,\quad \theta_{max}=\max_{k\in\mathcal{A}}\theta_k.
\]
For any $i\in \mathcal{B}\backslash\mathcal{A}$, it suffices to prove that
\begin{equation}\label{bdddiff}
\sup_{0 \leq t < \infty} |\theta_i(t)-\theta_{min}(t)|<\infty\quad\mbox{and}\quad \sup_{0 \leq t < \infty} |\theta_i(t)-\theta_{max}(t)|<\infty.
\end{equation}
We consider the following two cases.

\vspace{0.3cm}
\noindent $\diamond$ Case A. There exists a time $T>0$ and an integer $n\in\bbz$ for which 
\[ \theta_i(T)\in [\theta_{min}(T)+2n\pi,\theta_{max}(T)+2n\pi].  \]
In this case, we consider the solution $\tilde{\Theta}$ to \eqref{Ku} with initial data
\[
\tilde{\theta}_k^0=\theta_k(T)~\mbox{for}~k\neq i,\quad \mbox{and}\quad \tilde{\theta}_i^0=\theta_i(T)-2n\pi,
\]
and apply Theorem \ref{T3.1} (1) to the alternative ensemble $\tilde{\mathcal{A}}=\mathcal{A}\cup \{i\}$. We then have $D(\tilde{\Theta}_{\tilde{\mathcal{A}}}(t))\le\ell$ for $t\ge 0$, which translates back to
\[
 |\theta_i(t)-\theta_{min}(t)|\le 2|n|\pi+\ell\quad\mbox{and}\quad  |\theta_i(t)-\theta_{max}(t)|\le 2|n|\pi+\ell\quad \mbox{for }t\ge T.
\]
Thus \eqref{bdddiff} is true.

\vspace{0.3cm}
\noindent $\diamond$ Case B. There exists an integer $n\in\bbz$ such that 
\[ \theta_i(t)\in (\theta_{max}(T)+2n\pi,\theta_{min}(T)+2(n+1)\pi) \quad \mbox{for all $t\ge 0$}. \]
In this case, \eqref{bdddiff} is obvious.

\vspace{0.3cm}

\noindent (3)~(Well-ordering of $\Theta_{{\mathcal A}}$): For $i,j\in\mathcal{A}$ we repeat the calculation done in Lemma \ref{gronwall0}:
\begin{align}\label{individualdiff}
\begin{aligned}
\dot{\theta}_i-\dot{\theta}_j&=\nu_i-\nu_j+\frac{\kappa}{N}\sum_{k=1}^N\left[\sin(\theta_k-\theta_i)-\sin(\theta_k-\theta_j)\right]\\
&= \nu_i-\nu_j-2 \kappa \sin \Big( \frac{\theta_i-\theta_j}{2} \Big) \cdot\frac{1}{N}\sum_{k=1}^N \cos\left(\theta_k-\frac{\theta_i+\theta_j}{2}\right),
\end{aligned}
\end{align}
where
\begin{equation}\label{crude-mf0}
\frac{1}{N}\sum_{k=1}^N \cos\left(\theta_k-\frac{\theta_i+\theta_j}{2}\right)\le 1
\end{equation}
and
\begin{align}\label{crude-mf}
\begin{aligned}
&\frac{1}{N}\sum_{k=1}^N \cos\left(\theta_k-\frac{\theta_i+\theta_j}{2}\right) \\
& \hspace{0.5cm} = \frac{1}{N}\sum_{k\in\mathcal{A}} \cos\left(\theta_k-\frac{\theta_i+\theta_j}{2}\right)+\frac{1}{N}\sum_{k\in\mathcal{N}\backslash\mathcal{A}} \cos\left(\theta_k-\frac{\theta_i+\theta_j}{2}\right)\\\
&\hspace{0.5cm} \ge\frac{|\mathcal{A}|}{N}\cos D(\Theta_\mathcal{A})-\frac{N-|\mathcal{A}|}{N} \ge \gamma\cos D(\Theta_\mathcal{A})-(1-\gamma).
\end{aligned}
\end{align}
(Since we are not assuming that $\theta_i$ and $\theta_j$ are extremal as in Lemma \ref{gronwall0}, we cannot say that $\left|\theta_k-\frac{\theta_i+\theta_j}{2}\right|\le \frac{D(\Omega_\mathcal{B})}{2}$. We will have to be content with $\left|\theta_k-\frac{\theta_i+\theta_j}{2}\right|\le D(\Omega_\mathcal{B})$.) We would like \eqref{crude-mf} to be positive, and by Theorem \ref{T3.1}, this would roughly be true if $\phi_1<\cos^{-1}\frac{1-\gamma}{\gamma}$. 

In fact, since we have $\cos^{-1}\frac{1-\gamma}{\gamma}\le\theta_*$ by Lemma \ref{f-estimates}, and since $f$ is strictly increasing on $[0,\theta_*]$, we obtain
\begin{align*}
\begin{aligned}
\phi_1< \cos^{-1} \Big( \frac{1-\gamma}{\gamma} \Big)  \quad &\Longleftrightarrow \quad  f(\phi_1)< f(\cos^{-1}\Big(\frac{1-\gamma}{\gamma} \Big) )\\
&\Longleftrightarrow \quad  \frac{D(\Omega_\mathcal{B})}{\kappa}< \frac{(2\gamma-1)^{3/2}}{\sqrt{2\gamma}}\frac{2-\gamma}{\sqrt{\gamma/2}+(1-\gamma)} \quad (\because \mbox{Lemma }\ref{f-estimates}~(ii)).
\end{aligned}
\end{align*}
This is exactly condition \eqref{additional0}, and so we see that $\phi_1<\cos^{-1}\frac{1-\gamma}{\gamma}$ is true. The positivity of \eqref{crude-mf} is then in the following sense: for any $\psi_1\in(\phi_1,\cos^{-1}\frac{1-\gamma}{\gamma})$, we infer by Theorem \ref{T3.1} (1) that there exists a finite time $T>0$ such that
\begin{equation}\label{trapped}
D(\Theta_\mathcal{A}(t))\le \psi_1,\quad \forall t\ge T,
\end{equation}
and so, by \eqref{crude-mf},
\begin{equation}\label{crude-mf2}
\frac{1}{N}\sum_{k=1}^N \cos\left(\theta_k-\frac{\theta_i+\theta_j}{2}\right)\ge \gamma\cos D(\Theta_\mathcal{A})-(1-\gamma)\ge \gamma\cos \psi_1-(1-\gamma)>0,\quad\forall t\ge T.
\end{equation}
With the estimates \eqref{crude-mf0} and \eqref{crude-mf2} on the mean-field term at our disposal, we are now ready to prove the desired statement. \newline

We consider two cases:
\[ \mbox{Either} \quad \nu_i = \nu_j \quad \mbox{or} \quad \nu_i > \nu_j. \]

\noindent$\diamond$ Case A ($\nu_i=\nu_j$): In this case, \eqref{individualdiff} becomes
\[
\dot{\theta}_i-\dot{\theta}_j=- \frac{2\kappa}{N}\sum_{k=1}^N \cos\left(\theta_k-\frac{\theta_i+\theta_j}{2}\right)\cdot\sin\frac{\theta_i-\theta_j}{2},
\]
and thus for $t\ge T$ we have
\begin{align*}
\frac{d}{dt}|\theta_i-\theta_j|&=-\frac{2\kappa}{N}\sum_{k=1}^N \cos\left(\theta_k-\frac{\theta_i+\theta_j}{2}\right)\cdot \textrm{sgn}(\theta_i-\theta_j)\sin\frac{\theta_i-\theta_j}{2}\\
&=-\frac{2\kappa}{N}\sum_{k=1}^N \cos\left(\theta_k-\frac{\theta_i+\theta_j}{2}\right)\cdot \sin\frac{|\theta_i-\theta_j|}{2}\\
&\le -2\kappa (\gamma\cos\psi_1-(1-\gamma))\sin\frac{|\theta_i-\theta_j|}{2}~(\because \eqref{crude-mf2})\\
&\le -2\kappa (\gamma\cos\psi_1-(1-\gamma))\cdot\frac{1}{\pi}|\theta_i-\theta_j|~(\because 0\le |\theta_i-\theta_j|\le\pi)).
\end{align*}
Hence $|\theta_i-\theta_j|$ converges exponentially to zero.

\vspace{0.3cm}
\noindent$\diamond$ Case B. $\nu_i>\nu_j$: Recall \eqref{trapped}, which implies 
\[ |\theta_i(t)-\theta_j(t)|\le \psi_1 \quad \mbox{for $t\ge T$}. \]
Thus, by \eqref{individualdiff} and \eqref{crude-mf2}, we have
\[
\dot{\theta}_i(t_0)-\dot{\theta}_j(t_0)\ge \nu_i-\nu_j>0,\quad \mbox{for all }t_0\ge T\mbox{ such that } -\psi_1\le\theta_i(t_0)-\theta_j(t_0)\le 0.
\]
Therefore, by an exit-time argument, we may find a time $T_1\ge T$ such that
\[
0\le \theta_i(t)-\theta_j(t)\le \psi_1,\quad \forall t\ge T_1.
\]
Now, if we use \eqref{crude-mf0} along with the fact that $0\le \sin x\le x$ for $x\in [0,\frac{\pi}{4}]$, then \eqref{individualdiff} becomes
\begin{equation}\label{gronwall5}
\dot{\theta}_i-\dot{\theta}_j\ge\nu_i-\nu_j- \kappa (\theta_i-\theta_j),\quad t\ge T_1,
\end{equation}
and, analogously, if we use \eqref{crude-mf2} along with the fact that $0\le \frac{2\sqrt{2}}{\pi}x\le \sin x$ for $x\in[0,\frac{\pi}{4}]$, then \eqref{individualdiff} becomes
\begin{equation}\label{gronwall6}
\dot{\theta}_i-\dot{\theta}_j\le \nu_i-\nu_j-\frac{2\sqrt{2}}{\pi}(\gamma\cos\psi_1-(1-\gamma)) \kappa (\theta_i-\theta_j).
\end{equation}
By solving the Gronwall inequalities \eqref{gronwall5} and \eqref{gronwall6}, we may easily conclude that
\[
\frac{\nu_i-\nu_j}{\kappa}\le \liminf_{t\rightarrow\infty}(\theta_i(t)-\theta_j(t))\le \limsup_{t\rightarrow\infty}(\theta_i(t)-\theta_j(t))\le \frac{\pi}{2\sqrt{2}(\gamma\cos\psi_1-(1-\gamma))}\frac{\nu_i-\nu_j}{\kappa}.
\]
Recalling that $\psi_1\in(\phi_1,\cos^{-1}\frac{1-\gamma}{\gamma})$ was arbitrary, we conclude
\[
\frac{\nu_i-\nu_j}{\kappa}\le \liminf_{t\rightarrow\infty}(\theta_i(t)-\theta_j(t))\le \limsup_{t\rightarrow\infty}(\theta_i(t)-\theta_j(t))\le \frac{\pi}{2\sqrt{2}(\gamma\cos\phi_1-(1-\gamma))}\frac{\nu_i-\nu_j}{\kappa}.
\]
This completes the proof of Theorem \ref{T3.1}.

\subsection{A three-oscillator system} \label{sec:4.3}
With Theorem \ref{T3.1} and consequently Corollary \ref{C3.1} at our disposal, we can now prove complete phase-locking in the case $N =3$:
\begin{align}
\begin{aligned} \label{t-o}
{\dot \theta}_1 = \nu_1 + \frac{\kappa}{3}\Big(  \sin(\theta_2 - \theta_1) + \sin(\theta_3 - \theta_1)  \Big),  \\
{\dot \theta}_2 = \nu_2 + \frac{\kappa}{3}\Big(  \sin(\theta_1 - \theta_2) + \sin(\theta_3 - \theta_2)  \Big),   \\
{\dot \theta}_3 = \nu_3 + \frac{\kappa}{3}\Big(  \sin(\theta_1 - \theta_3) + \sin(\theta_2 - \theta_3)  \Big).
\end{aligned}
\end{align}
\begin{proposition}\label{P4.1}
Suppose that the coupling strength $\kappa$ satisfies 
\[ \kappa >\frac{\sqrt{138+22\sqrt{33}}}{4}D(\Omega)\approx 4.0649D(\Omega).
\]
Then, the Kuramoto flow $\Theta(t)$ exhibits complete phase-locking.
\end{proposition}
\begin{proof}
Let $\Theta^0 = (\theta_1^0, \theta_0^2, \theta_0^3)$ be an arbitrary initial phase vector, and let $\Theta = \Theta(t)$ be the solution to \eqref{t-o} with initial phase vector $\Theta^0$. \newline

\noindent $\bullet$ Case A: Suppose that there is no collision in the time evolution of $\Theta$. In this case, it follows from Theorem \ref{T2.1} that the solution $\Theta$ approaches a phase-locked state asymptotically. \newline

\noindent $\bullet$ Case B: Suppose that there are collisions. In this case, we may assume that that $\theta_1$ and $\theta_2$ collide at $t = t_0 \geq 0$ by possibly rearranging the indices. We may as well perform a time-shift so that $t_0=0$ is a time of collision, and shift $\theta_2$ modulo $2\pi$ so that
\begin{equation} \label{C-1}
\theta^0_1=\theta^0_2.
\end{equation}
Then, it follows from Corollary \ref{C3.1} that if we set $\gamma_0 = \frac{2}{3}$ and choose $\ell$ and $\kappa$ such that
\[ \ell \in \left(0,\frac{2\pi}{3}\right), \qquad |\theta^0_1-\theta^0_2|< \ell, \qquad  \kappa> \frac{D(\Omega)}{\frac{2}{3}(\sin \ell -\sin\frac{\ell}{2})} =: \kappa\left(\frac{2}{3},\ell, D(\Omega)\right), \]
then there is complete phase-locking. By \eqref{C-1},  we have
\[ |\theta^0_1-\theta^0_2 |=0, \]
so 
\[
\kappa >\min_{\ell \in(0,2\pi/3)} \kappa\left(\frac{2}{3},\ell, D(\Omega)\right)
\]
is a sufficient condition for asymptotic phase-locking. By direct calculation,
\[ \min_{\ell \in(0,2\pi/3)} \kappa_*\left(\frac{2}{3},\ell, D(\Omega)\right) = \min_{l\in(0,2\pi/3)}\frac{D(\Omega)}{\frac{2}{3}(\sin l-\sin\frac{l}{2})}=\frac{\sqrt{138+22\sqrt{33}}}{4}D(\Omega), \]
where the above minimization occurs at $\ell_* \in (0,2\pi/3)$ with $\cos \frac{\ell_*}{2} = \frac{1 + \sqrt{33}}{8}$. 
\end{proof}


\section{Complete phase-locking} \label{sec:5}
\setcounter{equation}{0}
In this section, we complete the discussion of subsection \ref{sec:3.2} by providing a proof of Theorem \ref{T3.2}. Suppose that the coupling strength and the initial data satisfy
 \begin{equation} \label{D-1}
 R_0 > 0 \quad \mbox{and} \quad   \kappa >1.6\frac{D(\Omega)}{R_0^2}.    
\end{equation} 
We will show that asymptotic phase-locking occurs with the additional features mentioned in the statement of Theorem \ref{T3.2}.
 This improves the earlier results (Proposition \ref{P2.2} and Theorem \ref{T2.3}) in \cite{H-K-R} because the lower bound for $\kappa$ appearing in the second relation is independent of $N$. Since the proof is rather long, we split it into four steps: \newline
\begin{itemize}
\item
Step A (emergence of a well-prepared phase): There exists an instant $t_0 \in [0, \infty)$ such that 
\[ R(t_0)\geq R_0 \quad \mbox{and}  \quad \Delta(t_0) \leq \frac{1}{4} \Big( \frac{D(\Omega)}{\kappa R_0} \Big)^2, \]
where $\Delta$ is the functional in \eqref{New-1}.  This follows from a careful investigation of the growth behavior of $R$, namely Lemma \ref{L2.1}.

\vspace{0.2cm}

\item
Step B (emergence of majority ensembles): We show that the phase described in Step A has a majority ensemble concentrated on a short arc. Specifically, we give quantitative conditions for $\gamma$ and $\ell$ such that $\Theta(t_0)$ has a $\gamma$-ensemble of arclength $\le \ell$.
\vspace{0.2cm}

\item
Step C (minimizing the lower bound for the coupling strength): We provide a choice for $\gamma$ and $\beta$ that satisfies the conditions of Step B as well as the conditions of Corollary \ref{C3.1}, in such a way that the required coupling strength $\kappa$ is the one given in \eqref{D-1}.

\vspace{0.2cm}

\item
Step D (harvesting all estimates): We collect all estimates in Step A - Step C to finish the proof.
\end{itemize}

\subsection{Step A (From disordered phases to well-prepared phases)} The idea is as follows. We begin with Lemma \ref{L2.1}:
\[
\dot{R}\ge \kappa \sqrt{\Delta}\left(R\sqrt{\Delta}-\frac{D(\Omega)}{2\kappa}\right).
\]
We can multiply $2R$ on both sides to see
\[
\frac{dR^2}{dt} \ge 2\kappa \cdot R\sqrt{\Delta}\left(R\sqrt{\Delta}-\frac{D(\Omega)}{2\kappa}\right).
\]
Then, for any $\epsilon>0$ we cannot have 
\[ R\sqrt{\Delta}>(1+\varepsilon)\frac{D(\Omega)}{2\kappa} \quad \mbox{for all $t\ge 0$}, \]
because otherwise we will have
\[ \frac{dR^2}{dt} \ge 2\kappa \cdot R\sqrt{\Delta}\left(R\sqrt{\Delta}-\frac{D(\Omega)}{2\kappa}\right) \geq 2\kappa \varepsilon (1 + \varepsilon) \Big( \frac{D(\Omega)}{2\kappa}  \Big)^2 = \mbox{positive constant},   \]
and so  $R^2 \to \infty$ as $t\rightarrow\infty$, which contradicts the fact that $0\le R \leq 1$. Hence, there must be an instant when $R\sqrt{\Delta}\le (1+\varepsilon)\frac{D(\Omega)}{2\kappa}$, or equivalently $\Delta \leq \frac{(1+\epsilon)^2}{4} \Big( \frac{D(\Omega)}{\kappa R} \Big)^2$. We can work a bit harder on this line of reasoning, by considering the growth behavior of $R$ and by using a more refined exit-time argument, to obtain the following estimate:

\begin{lemma}\label{L5.1}
Let $\Theta = \Theta(t)$ be the solution to \eqref{Ku} with initial data $\Theta^0$ and $\kappa > 0$. Then, there exists a time $t_0 \in [0, \infty)$ such that
\begin{equation} \label{D-2}
 R(t_0)\geq R_0 > 0 \quad \mbox{and}  \quad \Delta(t_0) \leq \frac{1}{4} \Big( \frac{D(\Omega)}{\kappa R_0} \Big)^2. 
\end{equation}
\end{lemma}
\begin{proof} For a given initial data $\Theta^0$, we have the following two cases:
\[  \mbox{Either} \quad \Delta(0) \leq  \frac{1}{4} \Big( \frac{D(\Omega)}{\kappa R_0} \Big)^2 \quad \mbox{or} \quad  \Delta(0) >  \frac{1}{4} \Big( \frac{D(\Omega)}{\kappa R_0} \Big)^2.   \]

\vspace{0.2cm}

\noindent $\bullet$ Case A: Suppose that 
\[  \Delta(0) \leq  \frac{1}{4} \Big( \frac{D(\Omega)}{\kappa R_0} \Big)^2. \]
In this case, we set $t_0  = 0$ to derive the desired estimate \eqref{D-2}.  \newline

\vspace{0.2cm}

\noindent $\bullet$ Case B: Suppose that 
\begin{equation} \label{D-3}
 \Delta(0) >  \frac{1}{4} \Big( \frac{D(\Omega)}{\kappa R_0} \Big)^2. 
\end{equation}
In this case, we define a set ${\mathcal T}$:
\[
\mathcal{T}=\left\{T>0: R(t)>0 \mbox{ and }\Delta(t) >  \frac{1}{4} \Big( \frac{D(\Omega)}{\kappa R(t)} \Big)^2~\textrm{for}~t\in[0,T)\right\}.
\]
The condition ``$R(t)>0$" is meant to guarantee the well-definedness and smoothness of $\Delta$.
By \eqref{D-3} and the continuity of $\Delta(\cdot)$, there exists some $\delta > 0$ such that $\delta \in {\mathcal T}$, which yields the nonemptiness of 
${\mathcal T}$. We now set 
\[
T^*=\sup \mathcal{T}>0.
\]
Then by definition of the set ${\mathcal T}$, it is clear that
\[  \mathcal{T}=(0,T^*). \]
On this interval, we may use Lemma \ref{L2.1} to obtain
\begin{align}\label{D-3-1}
\begin{aligned}
\dot{R}(t) &\geq \kappa \sqrt{\Delta(t)}\left[R \sqrt{\Delta(t)} -\frac{D(\Omega)}{2\kappa}\right]  > \kappa \sqrt{\Delta(t)}\left[R\left(\frac{D(\Omega)}{2\kappa R}\right)-\frac{D(\Omega)}{2\kappa}\right] =0, \quad t \in (0, T^*),
\end{aligned}
\end{align}
i.e., $R$ is strictly increasing on $(0,T^*)$. Now we have two cases, namely
\[
T^*<\infty\quad \mbox{or}\quad T^*=\infty.
\]

\vspace{0.2cm}

\noindent $\diamond$ Case B.1: Suppose in addition to \eqref{D-3} that $T^*< \infty.$ Since $R(t)$ is strictly increasing on $(0,T^*)$, we have
\begin{equation}\label{D-4-0}
R(T^*)>R_0>0.
\end{equation}
Now by definition of $\mathcal{T}$ and the finiteness of $T^*$, we have
\[\Delta(T^*) = \frac{D(\Omega)^2}{4\kappa^2 R(T^*)^2} \stackrel{\eqref{D-4-0}}{<} \frac{D(\Omega)^2}{4\kappa^2 R_0^2}.   \]
By setting $t_0 = T^*$, we obtain the desired estimate.  

\vspace{0.2cm}

\noindent $\diamond$ Case B.2: We are left with the most complicated case: suppose that we have \eqref{D-3} along with $T^*=\infty.$ Let us revisit inequality \eqref{D-3-1}, but with a slight twist:
\begin{align}\label{D-4-1}
\begin{aligned}
\dot{R}(t) &\geq \kappa \sqrt{\Delta(t)}\left[R \sqrt{\Delta(t)} -\frac{D(\Omega)}{2\kappa}\right] > \kappa \left(\frac{D(\Omega)}{2\kappa R}\right)\left[R\sqrt{\Delta(t)}-\frac{D(\Omega)}{2\kappa}\right]\\
&> \kappa \left(\frac{D(\Omega)}{2\kappa R}\right)\left[R\left(\frac{D(\Omega)}{2\kappa R}\right)-\frac{D(\Omega)}{2\kappa }\right] =0, \quad t \in (0, \infty).
\end{aligned}
\end{align}
The `twist' lies in changing the order in which the inequality $\sqrt{\Delta(t)}>\frac{D(\Omega)}{2KR}$ is applied(the second inequality in \eqref{D-4-1} is true because we already know that $\left[R \sqrt{\Delta(t)} -\frac{D(\Omega)}{2\kappa}\right]>0$). We rewrite the first, third and final terms of \eqref{D-4-1}:
\[
\dot{R} >\frac{D(\Omega)}{2}\left[\sqrt{\Delta(t)}-\frac{D(\Omega)}{2\kappa R}\right]
>0, \quad t \in (0, \infty).
\]
We integrate the above relation from $0$ to $\infty$, and use the boundedness of $R\in[0,1]$ to obtain
\[
1\ge \int_0^\infty \dot{R}(t)dt>\frac{D(\Omega)}{2}\int_0^\infty\left[\sqrt{\Delta(t)}-\frac{D(\Omega)}{2 \kappa R(t)}\right]dt.
\]
Hence, for any $\varepsilon>0$, we may find some time $\tau_\varepsilon>1$ such that
\[
\sqrt{\Delta(\tau_\varepsilon)}-\frac{D(\Omega)}{2\kappa R(\tau_\varepsilon)}<\varepsilon.
\]
Thus, if we choose $\varepsilon>0$ small enough so that
\[
\frac{D(\Omega)}{2\kappa R(1)}+\varepsilon<\frac{D(\Omega)}{2\kappa R_0},
\]
(this is possible because $R(1)>R(0)$ from \eqref{D-4-0}), then we get
\[
\sqrt{\Delta(\tau_\varepsilon)}<\frac{D(\Omega)}{2\kappa R(\tau_\varepsilon)}+\varepsilon\stackrel{\eqref{D-4-0}}{<}\frac{D(\Omega)}{2\kappa R(1)}+\varepsilon<\frac{D(\Omega)}{2\kappa R_0},
\]
or equivalently
\[
\Delta(\tau_\varepsilon)<\frac{D(\Omega)^2}{4\kappa^2 R_0^2}.
\] 
Hence
\[ R(\tau_\varepsilon) > R_0 \quad \mbox{and} \quad  \Delta(\tau_\varepsilon) < \frac{D(\Omega)^2}{4\kappa^2 R_0^2}.    \]
By setting $t_0 = \tau_\varepsilon$, we obtain the desired estimate.  
\end{proof}

\subsection{Step B (emergence of majority ensembles):} Next, we show that the condition \eqref{D-2} forces the phase vector $\Theta(t_0)$ prepared in Lemma \ref{L5.1} to contain some majority ensemble. The following is a heuristic argument. Each part of \eqref{D-2} means
\begin{align}
\begin{aligned} \label{D-5-0}
& \frac{1}{N}\sum_{j=1}^N \cos (\theta_j(t_0)-\phi(t_0))=R(t_0)\ge R_0, \\
& \frac{1}{N}\sum_{j=1}^N\sin^2(\theta_j(t_0)-\phi(t_0)) = \Delta(t_0) \le \frac{1}{4} \Big( \frac{D(\Omega)}{\kappa R_0} \Big)^2.
\end{aligned}
\end{align}
We can consider two cases depending on the size of $R_0$: \newline

\noindent $\diamond$ Case 1. If $R_0$ is sufficiently close to $1$, then \eqref{D-5-0}-(i) alone implies that many $(\theta_j-\phi)$'s lie close to $0$ mod $2\pi$, say in some fixed interval $(-\beta,\beta)$, and would form a majority ensemble concentrated on a small arc. The quantitative condition for these to form a $\gamma$-ensemble will be given in \eqref{D-5}-(i) below. \newline

\noindent $\diamond$ Case 2. On the other hand, if $R_0$ is significantly less than $1$ but $\kappa$ is sufficiently large so as to make the RHS of \eqref{D-5-0}-(ii) small, then on average $|\sin(\theta_j-\phi)|$ must be small, so for most $j$, $(\theta_j-\phi)$ must belong to either $(-\beta,\beta)$ or $(\pi-\beta,\pi+\beta)$ modulo $2\pi$. However, it follows from \eqref{D-5-0}-(i) that we should have more $(\theta_j-\phi)$'s in $(-\beta,\beta)$ rather than in $(\pi-\beta,\pi+\beta)$, so we can guess that the $(\theta_j-\phi)$'s lying in $(-\beta,\beta)$ should form a majority ensemble. The quantitative condition for this to happen will be given in \eqref{D-5}-(ii) below.

\begin{lemma}\label{L5.2}
Suppose that the following conditions hold.
\begin{enumerate}
\item
The phase vector $\Theta(t_0)$ satisfies \eqref{D-2}:
\begin{equation}\tag{\ref{D-2}}
R(t_0)\geq R_0 \quad \mbox{and}  \quad \Delta(t_0) \leq \frac{1}{4} \Big( \frac{D(\Omega)}{\kappa R_0} \Big)^2.
\end{equation}
\item
Parameters $\gamma\in(\frac{1}{2},1]$ and $\beta\in(0,\frac{\pi}{2})$ satisfy
\begin{equation}\label{D-5}
\mbox{either}\quad R_0\ge \gamma+(1-\gamma)\cos\beta\quad\mbox{or}\quad 2\gamma+\frac{D(\Omega)^2}{4\kappa^2R_0^2}\cdot \frac{1}{1-\cos\beta}\le 1+R_0.
\end{equation}
\end{enumerate}
Then, $\Theta(t_0)$ has a $\gamma$-ensemble of arclength $\le2\beta$.
\end{lemma}
\begin{proof}
We decompose the index set ${\mathcal N}=\{1,\cdots,N\}$ into two mutually disjoint subsets ${\mathcal A}$ and ${\mathcal B}$:
\begin{align}
\begin{aligned} \label{D-6}
\mathcal{A} &:=\left\{i=1,\cdots,N:\theta_i(t_0)\in[\phi(t_0)-\beta,\phi(t_0)+\beta]\mod 2\pi\right\}, \\
\mathcal{B} &:={\mathcal N}\backslash {\mathcal A}=\left\{i=1,\cdots,N:\theta_i(t_0)\in(\phi(t_0)+\beta,\phi(t_0)+2\pi-\beta)\mod 2\pi\right\}.
\end{aligned}
\end{align}
We claim that $\mathcal{A}$ is the desired $\gamma$-ensemble. Since \eqref{D-5} consists of two alternatives, let us work on each case separately.

\vspace{0.2cm}

\noindent $\bullet$ Case A: Suppose that $\eqref{D-5}_1$ holds, namely
\[
\gamma+(1-\gamma)\cos\beta \leq R_0.
\]
Then we have
\begin{align*}
R_0&\stackrel{\eqref{D-2}_1}{\le} R(t_0) \stackrel{\eqref{order-rel-2}}{=} \frac{1}{N} \sum_{j\in\mathcal{N}} \cos(\theta_j-\phi) = \frac{1}{N} \sum_{j\in\mathcal{A}} \cos(\theta_j-\phi)+\frac{1}{N} \sum_{j\in\mathcal{B}} \cos(\theta_j-\phi)\\
&\stackrel{\eqref{D-6}}{\le}\frac{1}{N} \sum_{j\in\mathcal{A}} 1+\frac{1}{N} \sum_{j\in\mathcal{B}} \cos\beta=\frac{|\mathcal{A}|}{N}+\frac{|\mathcal{B}|}{N}\cos\beta=\frac{|\mathcal{A}|}{N}+\left(1-\frac{|\mathcal{A}|}{N}\right)\cos\beta.
\end{align*}
We combine this with $\eqref{D-5}_1$ to get
\[
\gamma+(1-\gamma)\cos\beta\le R_0\le \frac{|\mathcal{A}|}{N}+\left(1-\frac{|\mathcal{A}|}{N}\right)\cos\beta,
\]
which gives
\[
\gamma(1-\cos\beta)\le \frac{|\mathcal{A}|}{N}(1-\cos\beta),
\]
and using $\cos\beta<1$, we conclude
\[
\gamma\le \frac{|\mathcal{A}|}{N}.
\]
Hence $\mathcal{A}$ is a $\gamma$-ensemble of arclength $\le2\beta$. \newline

\noindent $\bullet$ Case B: Now suppose that $\eqref{D-5}_2$ holds, namely
\[
 2\gamma+\frac{D(\Omega)^2}{4\kappa^2R_0^2}\cdot\frac{1}{(1-\cos\beta)}\le 1+R_0.
\]
Note that the map $\bbr\rightarrow\bbr$ given by $x\mapsto 1-x^2$ is concave, so by Jensen's inequality,
\[
1-x^2\ge 0\cdot \frac{\cos\beta-x}{1+\cos\beta}+\sin^2\beta\cdot \frac{1+x}{1+\cos\beta}=(1-\cos\beta)(1+x),\quad x\in[-1,\cos\beta].
\]
Substituting $x=\cos(\theta_i-\phi)\in[-1,\cos\beta]$ for each $i\in\mathcal{B}$ gives
\begin{equation}\label{D-7}
\sin^2(\theta_i-\phi) \ge (1-\cos\beta)(1+\cos(\theta_i-\phi)),\quad i\in\mathcal{B}.
\end{equation}
Therefore, we have
\begin{align*}
\begin{aligned}
\Delta(t_0)&\stackrel{\eqref{New-1}}{=}\frac{1}{N}\sum_{i\in\mathcal{N}} \sin^2 (\theta_i-\phi)\ge \frac{1}{N}\sum_{i\in\mathcal{B}} \sin^2 (\theta_i-\phi)\stackrel{\eqref{D-7}}{\ge} \frac{1}{N}\sum_{i\in\mathcal{B}} (1-\cos\beta)(1+\cos(\theta_i-\phi))\\
&=(1-\cos\beta)\frac{1}{N}\sum_{i\in\mathcal{B}} \left(1+\cos(\theta_i-\phi)\right)=(1-\cos\beta)\left[\frac{|\mathcal{B}|}{N}+\frac{1}{N}\sum_{i\in\mathcal{B}} \cos(\theta_i-\phi)\right]\\
&\stackrel{\eqref{order-rel-1}_1}{=}(1-\cos\beta)\left[\frac{N-|\mathcal{A}|}{N}+R_0-\frac{1}{N}\sum_{i\in\mathcal{A}} \cos(\theta_i-\phi)\right]\\
&\ge (1-\cos\beta)\left[\frac{N-|\mathcal{A}|}{N}+R_0-\frac{1}{N}\sum_{i\in\mathcal{A}} 1\right]= (1-\cos\beta)\left[\frac{N-|\mathcal{A}|}{N}+R_0-\frac{|\mathcal{A}|}{N}\right]\\
&= (1-\cos\beta)\left[1+R_0-\frac{2|\mathcal{A}|}{N}\right].
\end{aligned}
\end{align*}
Now
\begin{align*}
1+R_0&\stackrel{\eqref{D-5}_2}{\ge} 2\gamma+\frac{D(\Omega)^2}{4K^2R_0^2}\cdot\frac{1}{1-\cos\beta}\stackrel{\eqref{D-2}_2}{\ge} 2\gamma+\Delta(t_0)\cdot\frac{1}{1-\cos\beta}\ge 2\gamma+\left[1+R_0-\frac{2|\mathcal{A}|}{N}\right],
\end{align*}
and cancellation gives
\[
\frac{|\mathcal{A}|}{N}\ge \gamma.
\]
Thus, $\mathcal{A}$ is a $\gamma$-ensemble of arclength $\le 2\beta$. This concludes the proof.
\end{proof}

\begin{remark}\label{limitations}
1. A natural limitation of $\eqref{D-5}_1$ is that it requires
\[
R_0\ge \gamma+(1-\gamma)\cos\beta> \frac{1}{2}+0=\frac{1}{2}.
\]
Hence, for $R_0\le \frac{1}{2}$, we will have to use $\eqref{D-5}_2$. In fact, in Step C below, we will use $\eqref{D-5}_1$ for $R> 0.94$ and $\eqref{D-5}_2$ for $R\le 0.94$. \newline

\noindent 2. When $R_0\le \frac{1}{2}$, we can see from $\eqref{D-5}_2$ that we also have the natural limitation  $2\gamma\le 1+R_0$, that is $\gamma\le \frac{1+R_0}{2}$.
\end{remark}

\subsection{Step C (Minimizing the lower bound for the coupling strength)} Now, we want to apply the end result of Step B to Corollary \ref{C3.1}. To this end, we ask: 
\begin{quote}
Given $R_0>0$, can we choose $\gamma$, $\beta$ and $\kappa$ which, with $\ell=2\beta$, satisfy \eqref{D-5} as well as \eqref{B-1}? If so, how should we choose so that $\kappa$ is the smallest possible? In particular, why does the minimum possible $\kappa$ satisfies $\kappa >C\frac{D(\Omega)}{R_0^2}$ of Theorem \ref{T3.2}? Can we also satisfy \eqref{additional} as well?
\end{quote}

\noindent We can easily answer these questions for small $R_0$. We see from $\eqref{D-5}_2$ (recall from Remark \ref{limitations} that we must use $\eqref{D-5}_2$ instead of $\eqref{D-5}_1$ for small $R_0$) and the constraint $\gamma>\frac{1}{2}$ from $\eqref{B-1}_1$ that we need
\begin{equation}\label{D-10}
\gamma=\frac{1}{2}+{\mathcal O}(R_0),\qquad \frac{D(\Omega)^2}{4 \kappa^2R_0^2}\frac{1}{1-\cos\beta}={\mathcal O}(R_0).
\end{equation}
But then we would need from $\eqref{B-1}_2$ and $\ell=2\beta$ that
\[
\cos\beta>\frac{1-\gamma}{\gamma}=1- {\mathcal O}(R_0), \quad \sin\beta= {\mathcal O}(\sqrt{R_0}),
\]
so $\eqref{D-10}_2$ would require
\[
\frac{D(\Omega)^2}{4\kappa^2R_0^2}= {\mathcal O}(R_0^2), \quad \mbox{or equivalently} \quad \kappa >C\frac{D(\Omega)}{R_0^2}.
\]
Now $\eqref{B-1}_3$ requires
\begin{align*}
\kappa &>\frac{D(\Omega)}{2\sin\beta(\gamma\cos\beta-(1-\gamma))}=\frac{D(\Omega)}{2\sin\beta(\gamma(1+\cos\beta)-1)}\\
&\ge C\frac{D(\Omega)}{\sqrt{R_0}((0.5+O(R_0))(2-O(R_0))-1)}\ge C\frac{D(\Omega)}{R_0^{3/2}}.
\end{align*}
Hence, we see that we need at least $\kappa >C\frac{D(\Omega)}{R_0^2}$ in order to satisfy \eqref{B-1} and \eqref{D-5}. Furthermore, \eqref{additional} requires
\begin{align*}
\frac{D(\Omega)}{\kappa}< \frac{(2\gamma-1)^{3/2}}{\sqrt{2\gamma}}\frac{2-\gamma}{\sqrt{\gamma/2}+(1-\gamma)}=\frac{O(R_0)^{3/2}}{O(1)}\frac{O(1)}{O(1)}\le CR_0^{3/2},
\end{align*}
so, heuristically, \eqref{additional} is a weaker condition than \eqref{B-1} and \eqref{D-5} combined.
Conversely, it is not difficult to retrace the above argument to see that with suitable coefficients we can satisfy not only \eqref{B-1} and \eqref{D-5} but also \eqref{additional} with only $\kappa >C\frac{D(\Omega)}{R_0^2}$. \newline

The following lemma gives a reasonably optimal choice for $\gamma$, $\beta$ and $\kappa$.
\begin{lemma}\label{L5.3}
Suppose that  the initial data and coupling strength satisfy 
\[R_0 > 0\quad \mbox{and} \quad \kappa >1.6\frac{D(\Omega)}{R_0^2}. \]
Choose
\[
\gamma(R_0)=
\begin{cases}
0.5+\frac{0.35}{0.94}R_0& \mbox{if }0<R_0\le 0.94,\\
1-2.5(1-R_0)&\mbox{if } 0.94<R_0\le 1,
\end{cases}
\]
and
\[
\beta(R_0)=
\begin{cases}
\cos^{-1}(1-\frac{0.4}{0.94}R_0)&\mbox{if }0<R_0\le 0.94,\\
\cos^{-1}0.6&\mbox{if }0.94<R_0\le 1.
\end{cases}
\]
Then this choice of $\gamma$ and $\beta$ satisfies the assumption \eqref{D-5} of Lemma \ref{L5.2}, and the choice of $\gamma$ and $\ell=2\beta$ satisfies the conditions \eqref{B-1} and \eqref{additional} of Corollary \ref{C3.1}.
\end{lemma}
\begin{proof}
Since the proof is long and technical, we postpone it to Appendix A. 
\end{proof}
\begin{remark}
The above specification of $\gamma$ and $\beta$ were chosen to mimic the numerically best possible $\gamma(R_0)$ and $\beta(R_0)$ which minimize the required coupling strength $\kappa$. Numerical calculations suggest that the above constant 1.6 could be improved to 1.526.
\end{remark}

\subsection{Step D (Harvesting all estimates)}  Now Lemmas \ref{L5.1}, \ref{L5.2}, and \ref{L5.3} complete the proof of Theorem \ref{T3.2}. Suppose that  the initial data and coupling strength satisfy 
\[R_0 > 0\quad \mbox{and} \quad \kappa>1.6\frac{D(\Omega)}{R_0^2}. \]
By Lemma \ref{L5.1}, there exists a time $t_0\ge 0$ such that
\[
R(t_0)\geq R_0 \quad \mbox{and}  \quad \Delta(t_0) \leq \frac{1}{4} \Big( \frac{D(\Omega)}{\kappa R_0} \Big)^2. 
\]
Choose $\gamma(R_0)$ and $\beta(R_0)$ as in Lemma \ref{L5.3}. By Lemma \ref{L5.3}, $\gamma(R_0)$ and $\beta(R_0)$ satisfy \eqref{D-5}, and so, by Lemma \ref{L5.2}, $\Theta(t_0)$ has a $\gamma(R_0)$-ensemble of arclength $\le 2\beta(R_0)=:\ell$. But then, again by Lemma \ref{L5.3}, $\gamma(R_0)$ and $\ell$ satisfy \eqref{B-1} and \eqref{additional}. Hence, we conclude from Corollary \ref{C3.1} that:
\begin{enumerate}
\item  $\Theta(t)$ eventually has a stable $\gamma(R_0)$-ensemble of arclength $\le \frac{3\pi}{4(2\gamma(R_0)-1)}\frac{D(\Omega)}{\kappa}$;
\item complete phase-locking for $\Theta(t)$ occurs;
\item the aforementioned stable $\gamma(R_0)$-ensemble becomes well ordered in accordance with the natural frequencies.
\end{enumerate}


\section{Instability of dispersed states}\label{sec:6}
\setcounter{equation}{0}
In this section, we complete the discussion of subsection \ref{sec:3.3} by proving Theorem \ref{T3.3}. The idea is that the set of states with sufficiently small order parameters cannot be positively invariant under the Kuramoto flow \eqref{Ku} on $\bbt^N$.  \newline

For any $\varepsilon>0$ and $N$, we define a set
\[ {\mathcal D}(N, \varepsilon) := \Big \{\Theta \in \bbt^N:~
R(\Theta)<\frac{1}{\sqrt{N}}-\varepsilon \Big \}. \]
In subsection \ref{sec:6.1}, we will show that  $ {\mathcal D}(N, \varepsilon)$ is almost surely not positively invariant under the Kuramoto flow \eqref{Ku} with $\kappa > 0$, i.e., if the flow starts at a generic point in $ {\mathcal D}(N, \varepsilon)$, it cannot stay inside the region for  all time, regardless of the magnitude of the natural frequencies. In subsection \ref{sec:6.2}, as an application of this result, we prove Theorem \ref{T3.3} which yields a nontrivial upper bound for the pathwise critical coupling strength which is uniform with respect to initial data.

\subsection{The set ${\mathcal D}(N, \varepsilon)$ is not positively invariant}\label{sec:6.1} First, we begin with a heuristic argument for the fact that the set ${\mathcal D}(N, \varepsilon)$ is not positively invariant. For a given natural frequency vector $\Omega$,  consider the following two cases: \newline
\begin{itemize}
\item Case A ($\kappa \gg D(\Omega)$): According to numerical simulations, the flow tends to a phase-locked state contained in the interior of a quarter circle, and so the order parameter must eventually exceed $\cos \frac{\pi}{4}=\frac{1}{\sqrt{2}}\ge \frac{1}{\sqrt{N}}$.

\vspace{0.2cm}

\item Case B ($\kappa \ll D(\Omega)$, $N \gg 1$): In this case, the nonlinear coupling will contribute little to the dynamics of \eqref{Ku}, so we may approximate \eqref{Ku} as
\[
\dot{\theta}_i\approx\nu_i.
\]
Thus, the phases will tend to be randomly distributed over the unit circle and thus the Kuramoto order parameter will attain the temporal average
\begin{align*}
\begin{aligned}
& \langle R^2\rangle \stackrel{\eqref{order-rel-3}}{=}\frac{1}{N^2} \sum_{i,j} \langle\cos(\theta_i-\theta_j)\rangle=\frac{1}{N^2} \left(N+\sum_{i\neq j} \langle\cos(\theta_i-\theta_j)\rangle\right) \\
& \hspace{1cm} =\frac{1}{N^2} \left(N+\sum_{i\neq j} 0\right)=\frac{1}{N},
\end{aligned}
\end{align*}
where $\langle \cdot \rangle$ denotes the temporal average:
\[
    \lim_{T\rightarrow\infty}\frac{1}{T}\int_0^T R^2(t)dt=\frac{1}{N}.
\]
Thus, for any $\varepsilon>0$, such a solution cannot stay in the region ${\mathcal D}(N, \varepsilon)$ for all times $t\ge 0$.
\end{itemize}

Hence, it is reasonable to expect that given $\kappa >0,~\Omega=(\nu_1,\cdots,\nu
_N)$ and any $\varepsilon>0$, a generic solution to \eqref{Ku} cannot stay in the region ${\mathcal D}(N, \varepsilon)$ for all $t\ge 0$. Now, we will prove this heuristic argument rigorously using the divergence of the flow \eqref{Ku}. First, note that the Kuramoto model gives the integral curves to the following vector field:
\begin{equation}\label{vf-1}
    F: \mathbb{T}^N\rightarrow T\mathbb{T}^N,\quad F=(F_1,\cdots,F_N),
\end{equation}
where
\begin{equation}\label{vf-2}
    F_i(\Theta)=\nu_i+\frac{\kappa}{N}\sum_{j=1}^{N} \sin (\theta_j-\theta_i).
\end{equation}
Let us denote the flow of this vector field by $\Phi_t$, i.e., $\Phi_t(\Theta^0)=\Theta(t)$ is the solution to \eqref{Ku} with initial data $\Theta^0$. Also, for any subset $U\subset \bbt^N$ we denote
\[
\Phi_t(U)=\{\Phi_t(\Theta):\Theta\in U\}.
\]
Finally, we endow $\bbt^N$ with the standard Lebesgue measure $m$, normalized so that $m(\bbt^N)=(2\pi)^N$. Now the divergence of the vector field $F$ is
\begin{equation}\label{div-0}
    \nabla \cdot F=\sum_{i=1}^{N} \frac{\partial F_i}{\partial \theta_i}=-\frac{\kappa}{N}\sum_{1\leq i, j \leq N, i\neq j} \cos(\theta_j-\theta_i).
\end{equation}
On the other hand, it follows from \eqref{order-rel-3} that we also have
\begin{equation}\label{div-1}
    R^2=\frac{1}{N^2}\sum_{i,j= 1}^{N}\cos(\theta_j-\theta_i)=\frac{1}{N^2}\left(N+\sum_{1\leq i, j \leq N,i\neq j}\cos(\theta_j-\theta_i)\right).
\end{equation}
Thus, it follows from \eqref{div-0} and \eqref{div-1} that we can calculate the divergence as
\begin{equation}\label{div}
    \nabla \cdot F= \kappa (1-NR^2).
\end{equation}
Note that the divergence is independent of the choice of natural frequencies $\nu_i$, and depends only on $\kappa$, $N$ and $R$. 

With \eqref{div}, we can immediately show the instability of phase-locked states in ${\mathcal D}(N, \varepsilon)$. 
\begin{proposition} \label{P6.1}
The following two assertions hold. 
\begin{enumerate}
\item
If $\kappa>0$, then any phase-locked state of the Kuramoto model with order parameter less than $\frac{1}{\sqrt{N}}$ is unstable.
\item
If $\kappa<0$, then any phase-locked state of the Kuramoto model with order parameter greater than $\frac{1}{\sqrt{N}}$ is unstable.
\end{enumerate}
\end{proposition} 
\begin{proof}
If $\kappa>0$ and $\Theta$ is a phase-locked state with order parameter less than $\frac{1}{\sqrt{N}}$, then $\nabla\cdot F>0$ at $\Theta$. Thus if we linearize $F$ at $\Theta$, the Jacobian matrix has a positive trace and thus has a complex eigenvalue with positive real part. Hence, this equilibrium is unstable. The case for $\kappa<0$ can be treated similarly.
\end{proof}
More generally, we can rigorously prove the heuristic argument in the beginning of this subsection.
\begin{proposition}\label{nodispersed}
Let $\kappa$, $N$ and $\Omega=(\nu_1,\cdots,\nu_N)$ be fixed, and consider the flow $\Phi_t$ generated by \eqref{vf-1}-\eqref{vf-2}. Then, the following assertions hold:
\begin{enumerate}
\item If $\kappa>0$, let $R_*$ be any real constant so that $R_*<\frac{1}{\sqrt{N}}$. Then the Borel set
\[
U_{R_*}^\infty=\{ \Theta^0\in \bbt^N: R(\Phi_t(\Theta^0))<R_*\mbox{ for all } t\ge 0\}
\]
has measure zero.

\vspace{0.2cm}

\item If $\kappa <0$, let $R_*$ be any real constant so that $R_*>\frac{1}{\sqrt{N}}$. Then the Borel set
\[
V_{R_*}^\infty=\{ \Theta^0\in \bbt^N: R(\Phi_t(\Theta^0))>R_*\mbox{ for all } t\ge 0\}
\]
has measure zero.
\end{enumerate}
\end{proposition}
\begin{proof}
\noindent (i)~Let $\kappa >0$ and $R_*<\frac{1}{\sqrt{N}}$. If we define the open set
\[
U_{R_*}^0 :=\{\Theta^0\in\bbt^N: R(\Theta^0)<R_*\},
\]
then by \eqref{div} we have
\begin{equation}\label{posdiv}
(\nabla\cdot F)(\Theta)> \kappa (1-NR_*^2)>0,\quad\forall~\Theta\in U_{R_*}^0.
\end{equation}
For each $T> 0$, we define the open set
\[
U_{R_*}^T:=\{\Theta^0\in\bbt^N: R(\Phi_t(\Theta^0))<R_*\mbox{ for all }0\le t \le T\}.
\]
Then by definition we have
\[
\Phi_t(U_{R_*}^T)\subset U_{R_*}^0,\quad \forall~0\le t\le T.
\]
Hence, by \eqref{posdiv} and the fact that the divergence of a flow determines the rate of change of volume, we have
\[
\left.\frac{d}{dt}\right|^-m(\Phi_t(U_{R_*}^T))\ge \kappa(1-NR_*^2)\cdot m(\Phi_t(U_{R_*}^T)),\quad \forall~0\le t\le T,
\]
where $\left.\frac{d}{dt}\right|^-$ is the lower Dini derivative. In particular, we have
\[
m(\Phi_T(U_{R_*}^T))\ge \exp[\kappa(1-NR_*^2)T]m(\Phi_0(U_{R_*}^T))=\exp[\kappa(1-NR_*^2)T]m(U_{R_*}^T),
\]
or
\begin{align*}
\begin{aligned}
m(U_{R_*}^T) &\le m(\Phi_T(U_{R_*}^T))\exp[-\kappa(1-NR_*^2)T] \\
&\le m(\bbt^N)\exp[-\kappa(1-NR_*^2)T] = (2\pi)^N\exp[-\kappa(1-NR_*^2)T].
\end{aligned}
\end{align*}
Now
\[
U_{R_*}^\infty :=\mathop{\bigcap}_{T\in\bbn}U_{R_*}^T
\]
is a Borel set, and its measure satisfies
\[
m(U_{R_*}^\infty)\le\lim_{T\in\bbn,T\rightarrow\infty}m(U_{R_*}^T)\le\lim_{T\in\bbn,T\rightarrow\infty}(2\pi)^N\exp[-\kappa(1-NR_*^2)T]=0,
\]
as claimed. \newline

\noindent (ii)~The case $\kappa<0$ can be treated similarly as (i). Hence, we omit its details.
\end{proof}

\subsection{Proof of Theorem \ref{T3.3}}\label{sec:6.2} In this subsection, we provide a proof of Theorem \ref{T3.3} to obtain the finiteness of the pathwise critical coupling strength for fixed finite $N$. \newline

We choose a sequence $\{R_n\}_{n=1}^\infty$ of real numbers in $(0,\frac{1}{\sqrt{N}})$ sufficiently close to $\frac{1}{\sqrt{N}}$ such that
\[ \kappa >1.6\frac{D(\Omega)}{R_n^2} \quad \forall n\ge 1,\quad\mbox{and}\quad \lim_{n\rightarrow\infty}R_n=\frac{1}{\sqrt{N}}.
\]
Since $\kappa>0$, Proposition \ref{nodispersed} tells us that for every $n\ge 1$ the set
\[
U_{R_n}^\infty=\{ \Theta^0\in \bbt^N: R(\Phi_t(\Theta^0))<R_n\mbox{ for all } t\ge 0\}
\]
has measure zero and thus the union
\[
U^\infty=\mathop{\bigcup}_{n=1}^\infty U_{R_n}^\infty
\]
has measure zero as well. So, for almost all initial data $\Theta^0\in \bbt^N\backslash U^\infty$, for each $n\ge 1$ there exists a finite time $t_n\ge 0$ such that
\[
R(\Phi_{t_n}(\Theta^0))\ge R_n,
\]
and we apply Theorem \ref{T3.2} at this time $t=t_n$ to conclude the following assertions:
\begin{enumerate}
\item Complete phase-locking occurs.

\vspace{0.2cm}

\item Let
\[
\gamma(R_n)=0.5+\frac{0.35}{0.94}R_n\quad(\because R_n<\frac{1}{\sqrt{N}}<0.94).
\]
Then $\Theta(t)$ has a stable $\gamma(R_n)$-ensemble $\mathcal{A}_n$ of arclength $\le \frac{3\pi}{4(2\gamma(R_n)-1)}\frac{D(\Omega)}{\kappa}$; more precisely, after a suitable modulo $2\pi$-shift, we have
\[ \limsup_{t\rightarrow\infty}D(\Theta_\mathcal{A})\le \phi_1(\gamma(R_n),\kappa,D(\Omega)). \]

\vspace{0.2cm}

\item 
The ensemble $\mathcal{A}_n$ becomes well-ordered in accordance with the natural frequencies: if $\theta_i$ and $\theta_j$ belong to that ensemble and $\nu_i\ge \nu_j$, then
\begin{equation}\label{F-30}
\frac{\nu_i-\nu_j}{\kappa}\le \liminf_{t\rightarrow\infty}[\theta_i(t)-\theta_j(t)]\le \limsup_{t\rightarrow\infty}[\theta_i(t)-\theta_j(t)]
\le c_n\frac{\nu_i-\nu_j}{\kappa},
\end{equation}
where the constant $c_n$ is defined as
\[
c_n :=\frac{\pi}{2\sqrt{2}(\gamma(R_n)\cos\phi_1(\gamma(R_n), \kappa, D(\Omega))-(1-\gamma(R_n)))}.
\]
\end{enumerate} 
Now, since there are only finitely many subsets of $\mathcal{N}$, we may assume that the ensemble $\mathcal{A}_n=\mathcal{A}$ is fixed by possibly restricting to a subsequence of $\{R_n\}$ . But since
\[
\phi_1 \Big(\gamma(\frac{1}{\sqrt{N}}),\kappa,D(\Omega) \Big)<\frac{3\pi}{4(2\gamma_N-1)}\frac{D(\Omega)}{\kappa},
\]
and $\phi_1$ depends continuously on its first argument, we may find some natural number $n_0$ for which
\[
\phi_1(\gamma(R_{n_0}),\kappa,D(\Omega))<\frac{3\pi}{4(2\gamma_N-1)}\frac{D(\Omega)}{\kappa}
\]
and hence
\[
\limsup_{t\rightarrow\infty}D(\Theta_\mathcal{A})<\frac{3\pi}{4(2\gamma_N-1)}\frac{D(\Omega)}{\kappa}.
\]
Finally, we take the limit $n\rightarrow\infty$ in \eqref{F-30} to prove the final statement.


\section{Conclusion}\label{sec:7}
\setcounter{equation}{0}
In this paper, we introduced a pathwise critical coupling strength exhibiting a sharp transition from partial phase-locking to complete phase locking from a given initial phase configuration and provided nontrivial upper bounds for the pathwise critical coupling strength. A common paradigm, based on numerical simulations, for understanding the process of asymptotic phase-locking in the Kuramoto model  in the case $\kappa >D(\Omega)$ is as follows:

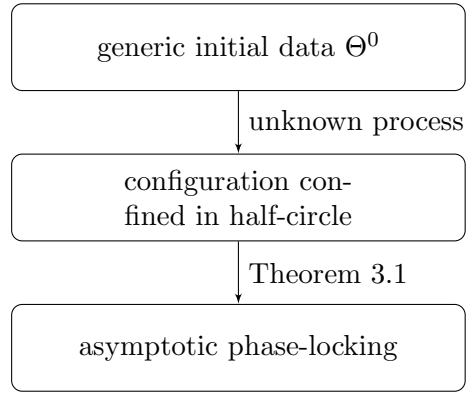
\begin{center}
\begin{tikzpicture}[node distance = 2cm, auto]

    \node [block] (gen) {generic initial data $\Theta^0$};
    \node [block, below of=gen] (half) {configuration confined in half-circle};
    \node [block, below of =half] (apl) {asymptotic phase-locking};
    \path [line] (gen) -- node {unknown process}(half);
    \path [line] (half) -- node {Theorem \ref{T3.1} } (apl);
\end{tikzpicture}
\end{center}
where the arrows denote the flow of time $t\ge 0$. That is, the flow \eqref{Ku} leads generic initial data into a half-circle by some unknown process, and then asymptotic phase-locking occurs via a well-understood process. However, so far there has been little progress towards understanding the first arrow. In this paper, we suggest that the following picture might be more amenable to mathematical analysis:
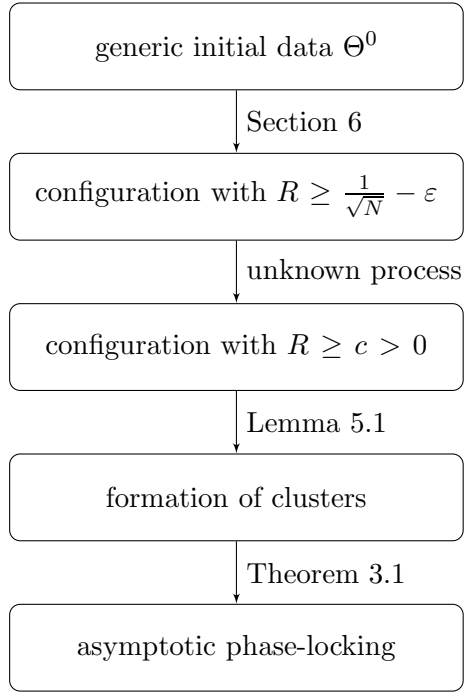
\begin{center}
\begin{tikzpicture}[node distance = 2cm, auto]
    \node [block] (gen) {generic initial data $\Theta^0$};
    \node [block, below of=gen] (div) {configuration with $R\ge \frac{1}{\sqrt{N}}-\varepsilon$};
    \node [block, below of=div] (R) {configuration with $R\ge c>0$};
    \node [block, below of=R] (clus) {formation of clusters};
    \node [block, below of =clus] (apl) {asymptotic phase-locking};

    \path [line] (gen) -- node{Section \ref{sec:6}} (div);
    \path [line] (div) -- node{unknown process} (R);
    \path [line] (R) -- node {Lemma \ref{L5.1}} (clus);
    \path [line]  (clus) -- node {Theorem \ref{T3.1}}(apl);
\end{tikzpicture}
\end{center}
That is, Section \ref{sec:6} tells us that we may assume the initial order parameter is at least $\frac{1}{\sqrt{N}}-\varepsilon$ (the first arrow), and once we have some lower bound on the order parameter (second arrow), we apply Theorem \ref{T3.3} (the last two arrows). In the paper, we have used $c=\max\{R_0,\frac{1}{\sqrt{N}}-\varepsilon\}$, which gave the bound
\[ \kappa_{pc}(\Theta^0,\Omega, N)\le 1.6 \min \Big \{\frac{1}{R_0^2},N \Big \}D(\Omega).
\]
If we could somehow prove that the order parameter must, with high probability, become eventually greater than some universal constant $c>0$, then we would obtain the bound
\[
\kappa_{pc}(\Theta^0,\Omega,N)\le \frac{1.6}{c^2}D(\Omega)
\]
which would prove Conjecture \ref{MainConj}. This is not completely unreasonable: the first arrow (that is, Section \ref{sec:6}) is valid even for any small positive coupling $\kappa$, so the assumptions of a stronger coupling $\kappa$ could improve the analysis of Section \ref{sec:6}. One method could be to replace the standard Lebesgue volume element on $\bbt^N$ by some other nontrivial one. It would be interesting to see this argument completed, and hence prove Conjecture \ref{MainConj}.

Another possible direction of research is towards the resulting phase-locked state. We have proved that a majority $\gamma$-ensemble must synchronize as described in Theorem \ref{T3.1}, but the quotient $\gamma$ has the natural limitation $\gamma\le \frac{1+R_0}{2}$ (as given in Remark \ref{limitations}). It would be nice to prove that the $\gamma$ approaches 1 as $\kappa$ is made large, which is indeed the case verified by numerical simulations, and it would be much better to prove that $\gamma$ actually becomes 1 for $\kappa$ above a certain threshold, such as that given in Conjecture \ref{AuxConj}. We leave such improvements toward Conjectures \ref{MainConj} and \ref{AuxConj} to be addressed in future works.

\appendix

\newpage

\appendix
\section{Proof of Lemma \ref{L5.3}}
\setcounter{equation}{0}
In this appendix, we provide a proof of Lemma \ref{L5.3}. The five conditions we must check are as follows.
\begin{align}
\begin{aligned} \label{Ap-1}
& (a)~\gamma \in\left(\frac{1}{2},1\right], \\
& (b)~ \beta\in\left(0,\cos^{-1}\Big[\frac{1}{\gamma}-1\Big]\right), \\
& (c)~\kappa >\frac{D(\Omega)}{2\sin\beta(\gamma\cos\beta-(1-\gamma))}, \\ 
& (d)~\mbox{either}\quad R_0\ge \gamma+(1-\gamma)\cos\beta\quad\mbox{or}\quad 2\gamma+\frac{D(\Omega)^2}{4\kappa^2R_0^2}\frac{1}{1-\cos\beta}\le 1+R_0,\\
& (e)~\frac{D(\Omega)}{\kappa}< \frac{(2\gamma-1)^{3/2}}{\sqrt{2\gamma}}\frac{2-\gamma}{\sqrt{\gamma/2}+(1-\gamma)}.
\end{aligned}
\end{align}
Recall our choice for $\gamma$ and $\beta$:
\[
\gamma(R_0)=
\begin{cases}
0.5+\frac{0.35}{0.94}R_0& 0<R_0\le 0.94,\\
1-\frac{5}{2}(1-R_0)& 0.94<R_0\le 1,
\end{cases}
\quad \mbox{and} \quad
\cos\beta(R_0)=
\begin{cases}
1-\frac{0.4}{0.94}R_0&0<R_0\le 0.94, \\
0.6&0.94<R_0\le 1.
\end{cases}
\]

\noindent We verify conditions \eqref{Ap-1}  for the two cases separately
\[  R_0\le 0.94 \quad \mbox{and} \quad R_0> 0.94. \]

\subsection{Case A$(R_0\le 0.94)$} We verify the conditions in \eqref{Ap-1} one by one. \newline

\noindent $\bullet$~Step 1 (Verification of relation \eqref{Ap-1}-(a)): Since $\gamma(R_0)$ is a linear function of $R_0$, we have
\[
0.5=\gamma(0)< \gamma(R_0)\le \gamma(0.94)=0.85.
\]

\vspace{0.2cm}

\noindent $\bullet$~Step 2 (Verification of relation \eqref{Ap-1}-(b)): First, note the equivalence:
\[
\mbox{\eqref{Ap-1}-(b)} \quad\Longleftrightarrow \quad \cos \beta(R_0)>\frac{1}{\gamma(R_0)}-1 \quad \Longleftrightarrow \quad \gamma(R_0)(1+\cos \beta(R_0))-1>0,
\]
On the other hand, we use the choices for $\gamma(R_0)$ and $\cos\beta(R_0)$ to see
\begin{align*}
\begin{aligned}
\gamma(R_0)(1+\cos \beta(R_0))-1&=\left(0.5+\frac{0.35}{0.94}R_0\right)\left(2-\frac{0.4}{0.94}R_0\right)-1=\frac{0.5}{0.94}R_0-\frac{0.14}{0.94^2}R_0^2\\
&=\frac{0.36}{0.94}R_0+\frac{0.14}{0.94}R_0\left(1-\frac{R_0}{0.94}\right)>0,
\end{aligned}
\end{align*}
where the final statement holds because of $0<R_0\le 0.94$.

\vspace{0.2cm}

\noindent $\bullet$~Step 3 (Verification of relation \eqref{Ap-1}-(c)): Note that
\begin{align*}
\begin{aligned}
\eqref{Ap-1}-(c)\quad\Longleftarrow&\quad 1.6\frac{D(\Omega)}{R_0^2}\ge \frac{D(\Omega)}{2\sin\beta(R_0)(\gamma(R_0)\cos\beta(R_0)-(1-\gamma(R_0)))}\\
\Longleftrightarrow&\quad 1.6\frac{1}{R_0^2}\ge \frac{1}{2\sqrt{1-\cos\beta(R_0)}\sqrt{1+\cos\beta(R_0)}(\gamma(R_0)(1+\cos\beta(R_0))-1)}\\
\Longleftrightarrow&\quad 1.6\frac{1}{R_0^2}\ge \frac{1}{2\sqrt{\frac{0.4}{0.94}R_0}\sqrt{2-\frac{0.4}{0.94}R_0}(\frac{0.5}{0.94}R_0-\frac{0.14}{0.94^2}R_0^2)}\\
\Longleftrightarrow&\quad 3.2\frac{1}{R_0^2}\sqrt{\frac{0.4}{0.94}R_0}\sqrt{2-\frac{0.4}{0.94}R_0}(\frac{0.5}{0.94}R_0-\frac{0.14}{0.94^2}R_0^2)\ge 1\\
\Longleftrightarrow&\quad 3.2\sqrt{\frac{0.4}{0.94}}\sqrt{2-\frac{0.4}{0.94}R_0}\left(\frac{0.5}{0.94\sqrt{R_0}}-\frac{0.14}{0.94^2}\sqrt{R_0}\right)\ge 1.
\end{aligned}
\end{align*}
In the final inequality, each and every term on the LHS is a  positive decreasing function of $R_0$ on $(0,0.94]$, so the LHS as a whole is a decreasing function of $R_0$ on $(0,0.94]$. At $R_0=0.94$, the LHS is 1.0430..., which is larger than 1. We thus conclude that \eqref{Ap-1}-(c) is true.

\vspace{0.2cm}

\noindent $\bullet$~Step 4 (Verification of relation \eqref{Ap-1}-(d)):  We use the second condition of \eqref{Ap-1}-(d).
\begin{align*}
\eqref{Ap-1}-(d)\quad \Longleftarrow &\quad 2\gamma(R_0)+\frac{D(\Omega)^2}{4\kappa^2R_0^2}\frac{1}{1-\cos\beta(R_0)}\le 1+R_0\\
\Longleftarrow &\quad 2\gamma(R_0)+\frac{R_0^2}{4\cdot 1.6^2}\frac{1}{1-\cos\beta(R_0)}\le 1+R_0\\
\Longleftrightarrow &\quad 2\left(0.5+\frac{0.35}{0.94}R_0\right)+\frac{R_0^2}{4\cdot 1.6^2}\frac{1}{\frac{0.4}{0.94}R_0}\le 1+R_0\\
\Longleftrightarrow &\quad \frac{0.7}{0.94}R_0+\frac{0.94}{4\cdot 1.6^2\cdot 0.4}R_0\le R_0\\
\Longleftrightarrow &\quad \frac{0.7}{0.94}+\frac{0.94}{4\cdot 1.6^2\cdot 0.4}=0.9742...\le 1.
\end{align*}
Hence \eqref{Ap-1}-(d) is true.

\vspace{0.2cm}

\noindent $\bullet$~Step 5 (Verification of relation \eqref{Ap-1}-(e)):  We have
\begin{align*}
\eqref{Ap-1}-(e)\quad \Longleftarrow & \quad\frac{1}{1.6}R_0^2\le \frac{(2\gamma(R_0)-1)^{3/2}}{\sqrt{2\gamma(R_0)}}\frac{2-\gamma(R_0)}{\sqrt{\gamma(R_0)/2}+(1-\gamma(R_0))}\\
\Longleftrightarrow &\quad \frac{1}{1.6}R_0^2\le \frac{\left(\frac{0.7}{0.94}R_0\right)^{3/2}}{\sqrt{1+\frac{0.7}{0.94}R_0}}\frac{1.5-\frac{0.35}{0.94}R_0}{\sqrt{0.25+\frac{0.35}{1.88}R_0}+0.5-\frac{0.35}{0.94}R_0}\\
\Longleftrightarrow &\quad\frac{1}{1.6}\left(\frac{0.94}{0.7}\right)^{3/2}\frac{\sqrt{R_0}\sqrt{1+\frac{0.7}{0.94}R_0}\left(\sqrt{0.25+\frac{0.35}{1.88}R_0}+0.5-\frac{0.35}{0.94}R_0\right)}{1.5-\frac{0.35}{0.94}R_0}\le 1.
\end{align*}
The LHS of the final inequality has logarithmic derivative
\begin{align*}
&\frac{1}{2R_0}+\frac{\frac{0.7}{0.94}}{2(1+\frac{0.7}{0.94}R_0)}+\frac{\frac{\frac{0.35}{1.88}}{2\sqrt{0.25+\frac{0.35}{1.88}R_0}}-\frac{0.35}{0.94}}{\sqrt{0.25+\frac{0.35}{1.88}R_0}+0.5-\frac{0.35}{0.94}R_0}+\frac{\frac{0.35}{0.94}}{1.5-\frac{0.35}{0.94}R_0}\\
& \hspace{0.5cm} \ge \frac{1}{2\cdot 0.94}+0+\frac{\frac{\frac{0.35}{1.88}}{2\sqrt{0.25+\frac{0.35}{1.88}\cdot 0.94}}-\frac{0.35}{0.94}}{\sqrt{0.25+\frac{0.35}{1.88}R_0}+0.5-\frac{0.35}{0.94}R_0}+0\\
& \hspace{0.5cm}= \frac{1}{2\cdot 0.94}-\frac{0.35}{0.94}\cdot \frac{1-\frac{1}{4\sqrt{0.425}}}{\sqrt{0.25+\frac{0.35}{1.88}R_0}+0.5-\frac{0.35}{0.94}R_0}\\
& \hspace{0.5cm}\ge \frac{1}{2\cdot 0.94}-\frac{0.35}{0.94}\cdot \frac{1-\frac{1}{4\sqrt{0.425}}}{\sqrt{0.25+\frac{0.35}{1.88}\cdot 0}+0.5-\frac{0.35}{0.94}\cdot 0.94}\\
& \hspace{0.5cm}=\frac{1}{2\cdot 0.94}-\frac{0.35}{0.94}\cdot \frac{1-\frac{1}{4\sqrt{0.425}}}{0.65}\\
& \hspace{0.5cm}=0.178...>0,
\end{align*}
and thus is an increasing function of $R_0$ on $(0,0.94]$. At $R_0=0.94$, this LHS has value 0.857..., which is less than $1$. Hence \eqref{Ap-1}-(e) is true.

\subsection{Case B $(R_0 > 0.94)$}  Again, we verify each relation one by one.  \newline

\noindent $\bullet$~Step 1 (Verification of relation \eqref{Ap-1}-(a)): Since $\gamma(R_0)$ is a linear function of $R_0$,
\[
0.85=\gamma(0.94)< \gamma(R_0)\le \gamma(1)=1.
\]

\noindent $\bullet$~Step 2 (Verification of relation \eqref{Ap-1}-(b)): Note the equivalence 
\[
\eqref{Ap-1}-(b)\quad \Longleftrightarrow \quad \cos \beta(R_0)>\frac{1}{\gamma(R_0)}-1 \Longleftrightarrow \quad \gamma(R_0)(1+\cos \beta(R_0))-1>0,
\]
and we can calculate
\begin{align*}
\gamma(R_0)(1+\cos \beta(R_0))-1&=\left(1-2.5(1-R_0)\right)\cdot 1.6-1=0.6-4(1-R_0)\\
&>0.6-4\cdot 0.06=0.36>0,
\end{align*}
where we have used $0.94<R_0\le 1$ in the first inequality. \newline

\noindent $\bullet$~Step 3 (Verification of relation \eqref{Ap-1}-(c)): We have 
\begin{align*}
\eqref{Ap-1}-(c)\quad \Longleftarrow&\quad 1.6\frac{D(\Omega)}{R_0^2}\ge \frac{D(\Omega)}{2\sin\beta(R_0)(\gamma(R_0)\cos\beta(R_0)-(1-\gamma(R_0)))}\\
\Longleftrightarrow &\quad1.6\frac{1}{R_0^2}\ge \frac{1}{2\cdot 0.8\cdot(0.6-4(1-R_0))}\\
\Longleftrightarrow &\quad2.56\frac{4R_0-3.4}{R_0^2}\ge 1.
\end{align*}
In the final inequality, the LHS has derivative
\[
2.56\left(-\frac{4}{R_0^2}+\frac{6.8}{R_0^3}\right)>2.56\left(-\frac{4}{0.94^2}+6.8\right)=2.2731...>0,
\]
and thus is an increasing function of $R_0$ on $(0.94,1]$. At $R_0=0.94$, the LHS is 1.0430..., which is larger than 1. We thus conclude that \eqref{Ap-1}-(c) is true. 

\vspace{0.2cm}

\noindent $\bullet$~Step 4 (Verification of relation \eqref{Ap-1}-(d)):  This time, we verify the first statement of \eqref{Ap-1}-(d), which is
\begin{align*}
\eqref{Ap-1}-(d)-(i)\quad \Longleftrightarrow &\quad R_0\ge \gamma(R_0)+(1-\gamma(R_0))\cos\beta(R_0)\\
\Longleftrightarrow &\quad R_0\ge (1-2.5(1-R_0))+2.5(1-R_0)\cdot 0.6\\
\Longleftrightarrow &\quad 0\ge 0.
\end{align*}

\vspace{0.2cm}

\noindent $\bullet$~Step 5 (Verification of relation \eqref{Ap-1}-(e)):  We have
\begin{align*}
\begin{aligned}
\eqref{Ap-1}-(e)\quad \Longleftarrow & \quad\frac{1}{1.6}R_0^2\le \frac{(2\gamma(R_0)-1)^{3/2}}{\sqrt{2\gamma(R_0)}}\frac{2-\gamma(R_0)}{\sqrt{\gamma(R_0)/2}+(1-\gamma(R_0))}\\
\Longleftrightarrow &\quad \frac{1}{1.6}R_0^2\le \frac{\left(5R_0-4\right)^{3/2}}{\sqrt{5R_0-3}}\frac{3.5-2.5R_0}{\sqrt{1.25R_0-0.75}+2.5(1-R_0)}\\
\Longleftrightarrow &\quad \frac{1}{1.6}\frac{R_0^2\sqrt{5R_0-3}(\sqrt{1.25R_0-0.75}+2.5(1-R_0))}{\left(5R_0-4\right)^{3/2}(3.5-2.5R_0)}\le 1.
\end{aligned}
\end{align*}
The LHS of the final inequality has, on $(0.94,1)$, logarithmic derivative
\begin{align*}
\begin{aligned}
&\frac{2}{R_0}+\frac{5}{2(5R_0-3)}+\frac{\frac{1.25}{2\sqrt{1.25R_0-0.75}}-2.5}{\sqrt{1.25R_0-0.75}+2.5(1-R_0)}-\frac{3}{2}\cdot\frac{5}{5R_0-4}+\frac{2.5}{3.5-2.5R_0}\\
&\hspace{0.5cm} \le \frac{2}{0.94}+\frac{5}{2(5\cdot 0.94-3)}+\frac{\frac{1.25}{2\sqrt{1.25\cdot 0.94-0.75}}-2.5}{\sqrt{1.25R_0-0.75}+2.5(1-R_0)}-\frac{3}{2}\cdot\frac{5}{5\cdot 1-4}+\frac{2.5}{3.5-2.5\cdot 1}\\
&\hspace{0.5cm}= \frac{2}{0.94}+\frac{5}{3.4}-\frac{2.5-\frac{1.25}{2\sqrt{0.425}}}{\sqrt{1.25R_0-0.75}+2.5(1-R_0)}-7.5+2.5\\
&\hspace{0.5cm}\le \frac{2}{0.94}+\frac{5}{3.4}-\frac{2.5-\frac{1.25}{2\sqrt{0.425}}}{\sqrt{1.25\cdot 1-0.75}+2.5(1-0.94)}-5\\
&\hspace{0.5cm}\le \frac{2}{0.94}+\frac{5}{3.4}-\frac{2.5-\frac{1.25}{2\sqrt{0.425}}}{\sqrt{0.5}+0.15}-5\\
&\hspace{0.5cm}=-3.20...<0,
\end{aligned}
\end{align*}
and thus is a decreasing function of $R_0$ on $(0.94,1]$. At $R_0=0.94$, this LHS has value 0.857..., which is less than $1$. Hence \eqref{Ap-1}-(e) is true.

\end{document}